\documentclass[11pt,a4paper,reqno]{amsart}

\usepackage{amsfonts,amsmath,amsthm,amssymb,mathrsfs}
\usepackage{mathtools}
\usepackage{enumitem}
\usepackage{lmodern,microtype}
\usepackage{orcidlink}
\usepackage{geometry}

\RequirePackage{xcolor}
\RequirePackage{hyperref}
\definecolor{colorLink}{RGB}{0,100,162}
\definecolor{colorCite}{RGB}{8,124,100}
\definecolor{colorURL}{RGB}{130,105,70}
\definecolor{brickred}{RGB}{178,34,34}
\definecolor{linkblue}{RGB}{0,71,171}
\hypersetup{colorlinks=true,allcolors=linkblue}

\mathtoolsset{showonlyrefs}
\numberwithin{equation}{section}
\allowdisplaybreaks
\usepackage{graphicx}

\theoremstyle{plain}
\newtheorem{theorem}{Theorem}[section]
\newtheorem{lemma}[theorem]{Lemma}
\newtheorem{proposition}[theorem]{Proposition}
\newtheorem{corollary}[theorem]{Corollary}        
 
\newtheorem{condition}[theorem]{Condition}  

\theoremstyle{definition}
\newtheorem{definition}[theorem]{Definition}
\newtheorem{example}[theorem]{Example}
\newtheorem{remark}[theorem]{Remark}

\DeclareMathOperator{\Esp}{E}
\DeclareMathOperator{\Prob}{P}

\DeclareMathOperator{\Var}{Var}
\DeclareMathOperator{\Qrob}{Q}
\DeclareMathOperator{\IR}{\mathbb{R}}
\DeclareMathOperator{\IN}{\mathbb{N}}
\DeclareMathOperator{\bF}{\mathcal{F}}
\DeclareMathOperator{\dom}{dom}

\newcommand{\wt}[0]{\widetilde}
\newcommand{\mc}[0]{\mathcal}
\newcommand{\msf}[0]{\mathsf}
\newcommand{\mfk}[0]{\mathfrak}
\newcommand{\law}[0]{\mathrm{law}}
\newcommand{\eqlaw}[0]{{\overset{\law}{=}}}

\newcommand{\cout}[0]{ \mathrel{\text{\bfseries out}} }
\newcommand{\cin}[0]{ \mathrel{\text{\bfseries in}} }
\newcommand{\sX}[0]{\wt \mX^{\Delta}}
\newcommand{\sV}[0]{V_{\Delta}}
\newcommand{\sE}[0]{E_{\Delta}}
\newcommand{\Lop}[0]{\mathrm L}
\newcommand{\Gop}[0]{\mathrm G}
\newcommand{\D}[0]{\mathrm D}
\newcommand{\mX}[0]{\mc X}
\newcommand{\mW}[0]{\msf W}
\newcommand{\mY}[0]{\mc Y}

\newcommand{\fp}[0]{\mfk p}
\newcommand{\ft}[0]{\mfk t}
\newcommand{\bv}[0]{\mathbf{v}}
\newcommand{\bz}[0]{\mathbf{z}}
\newcommand{\by}[0]{\mathbf{y}}
\newcommand{\bx}[0]{\mathbf{x}}
\newcommand{\bg}[0]{\mathbf{g}}
\newcommand{\rd}{\mathrm{d}}
\newcommand{\vd}{\,\mathrm{d}}
\newcommand{\process}[1]{(#1)_{t\ge 0}}
\newcommand{\indic}[1]{\mathbf{1}_{\{#1\}}}
\newcommand{\loct}[3]{L^{#2}_{#3}(#1)}

\newcommand{\braces}[1]{ \left({#1}\right) } 
\newcommand{\sqbraces}[1]{ \left[{#1}\right] } 
\newcommand{\cubraces}[1]{ \left\{{#1}\right\}} 
\newcommand{\xnorm}[2]{ \left\|{#1}\right\|_{#2}} 
\newcommand{\abs}[1]{ \left|{#1}\right|} 
\newcommand{\xabs}[2]{ \abs{#1}_{#2}}

\newenvironment{enuroman}{\begin{enumerate}[label=(\roman*)]}{\end{enumerate}}
\newenvironment{xenumerate}[1]{\begin{enumerate}[label= {\bfseries {#1}\arabic*.}, ref= { \bfseries  {#1}\arabic*}]}{\end{enumerate}}
	
\title[Diffusions on Metric Graphs via STMCA]{General Diffusions on Metric Graphs as Limits of Space-Time Markov Chains}

\author[A. Anagnostakis]{Alexis Anagnostakis}
\address{Centro de Modelamiento Matemático (CNRS IRL2807), Universidad de Chile, Santiago} 
\email{\url{aanagnostakis@cmm.uchile.cl}}  
\urladdr{\orcidlinkf{0009-0000-3306-9766}}

\date{October 2026}

\begin{document}
	
\begin{abstract}
We introduce the Space-Time Markov Chain Approximation (STMCA) for a general diffusion process on a finite metric graph $\Gamma$. 
The STMCA is a doubly asymmetric (in both time and space) random walk  defined on subdivisions of $\Gamma$, with transition probabilities and conditional transition times that match, in expectation, those of the target diffusion. 
We derive bounds on the $p$-Wasserstein distances between a non-explosive diffusion and its STMCA in terms of a thinness quantifier of the subdivision. An extension to absorbed diffusions uses a boundary-adapted Green-kernel thinness that remains finite even at certain singular exit boundaries. Under the stated regularity and subdivision hypotheses, this gives convergence at any rate strictly less than $\frac14\wedge\frac1p$ in the appropriate thinness quantifier. 
For suitably adapted
subdivisions satisfying a quadratic thinness bound, we show that the rate is doubled:
we achieve any rate strictly less
than $\frac12\wedge\frac2p$ in terms of the maximum
cell size.
Additionally, we provide explicit analytical formulas for transition probabilities and times, enabling practical implementation of the STMCA, as well as easily verifiable sufficient non-explosion conditions for the bounds to hold. 
Numerical experiments illustrate our results.
\end{abstract}
	
\maketitle

	
	\section{Introduction}
	\label{sec_introduction}

	Since their introduction in~\cite{lumer1980connecting} and later in~\cite{freidlin1993diffusion}, diffusions on metric graphs have attracted growing interest. 
	From a modeling perspective, they describe diverse physical phenomena, including electrical networks, nerve impulse propagation, and fluid motion in fissured porous media (see~\cite{freidlin1993diffusion,lejay2003simulating,weber2001occupation} and references therein). Theoretically, they represent a non-trivial class of diffusions that blend local one-dimensional dynamics with global geometric constraints.

More precisely, a diffusion on a metric graph is a continuous
strong Markov process that behaves like a one-dimensional
diffusion on each edge and, upon reaching a vertex, selects
an incident edge according to fixed probabilities known as
the \emph{spinning measure}. The process may also exhibit
sticky behavior at vertices, spending a positive amount
of time there.

Diffusions on metric graphs are characterized by an
edge-indexed family of generalized one-dimensional
second-order differential operators and a vertex-indexed
family of Wentzell-type transmission conditions that
blend first and (generalized) second-order derivatives.
Each differential operator describes the local dynamics
of the process on its corresponding edge, while each
transmission condition determines the spinning measure
and the occupation behavior at the vertex. In particular,
a positive coefficient of the second-order term gives
rise to a \emph{sticky vertex}, at which the process
spends positive occupation time.
To emphasize the general nature of these operators,
we refer to such processes as general diffusions
on metric graphs.

	Despite considerable recent advances---particularly for star-graphs, where the analysis is simpler (see~\cite{berry2024sticky,berry2025stationary,bobrowski2024snapping,lempa2024diffusion,martinez2025martingale,martinez2025comparison,MartinezOhavi2026,SalminenStenlund2025})---the development of numerical approximation methods for diffusions on metric graphs remains surprisingly understudied. 
	Existing approaches include:
	\begin{enumerate}
		\item \label{item_method1} Birth-death processes, see~\cite{weber2001occupation};
		\item \label{item_method2} Modified Euler scheme that handles nodal behavior by relocating the process to the boundary of a small vertex neighborhood with fixed probabilities upon entry, see~\cite{dassios2022firsthitting};
		\item \label{item_method3} The invariance principle from~\cite{pavlyukevich2024walsh} for simple fully symmetric random walks on star-graphs, whose limit is Walsh Brownian motion.
	\end{enumerate}
	
These approaches have limitations when applied to general diffusions on metric graphs. To the best of our knowledge, convergence rates have not been established for the first two methods. Moreover, the modified Euler scheme requires sufficiently regular edgewise dynamics to admit an Euler discretization, whereas the symmetric random-walk invariance principle applies specifically to Walsh Brownian motion.

To address these limitations, we introduce the
Space-Time Markov Chain Approximation (STMCA)
for general diffusions on finite metric graphs.
The STMCA is a random walk on a subdivision of
the graph, whose transition probabilities and
destination-dependent holding times match the
corresponding exit probabilities and conditional
mean exit times of the target diffusion.

The construction applies to general diffusions
characterized by edgewise scale functions,
speed measures, and vertex transmission
conditions. Our quantitative convergence
results, however, concern diffusions on
natural scale on each edge (NSE diffusions).
Under suitable regularity and subdivision
assumptions, we establish bounds on the
$p$-Wasserstein distance between the path-space
laws of the diffusion and its STMCA in terms
of a thinness quantifier determined by the
speed measures. For subdivisions
satisfying a quadratic thinness bound,
these estimates yield convergence at any
rate strictly less than
$\frac12\wedge\frac2p$ in terms of the
maximum cell size.

We extend the convergence analysis to diffusions with accessible absorbing endpoints. For this purpose, we introduce a boundary-adapted thinness quantifier based on the diagonal killed Green kernel, which remains finite near such endpoints even when the speed measure has infinite mass there.

The convergence proof extends the coupling
strategy developed for one-dimensional
diffusions in~\cite{anagnostakis2023general}.
Its main ingredients are the embedding
property of the STMCA, moment estimates for
the embedding times, and temporal regularity
bounds for the target diffusion. The
embedding property identifies the STMCA
jump skeleton with the successive positions
of the diffusion at subdivision vertices,
thereby providing a natural coupling.
To establish the required regularity
estimates, we first treat star-graph
diffusions using excursion-flipping and
time-change arguments, building on the
characterizations in~\cite{anagnostakis2025walsh}.
We then extend these estimates to finite
metric graphs by decomposing diffusion paths
into successive stopped star-graph segments
and applying the strong Markov property.

In addition, we derive explicit formulas
for the STMCA transition probabilities
and destination-conditional holding times
in terms of the scale functions, speed
measures, and vertex parameters. These
formulas make the scheme directly
implementable.

Our approach extends the STMCA for general one-dimensional diffusions developed in~\cite{anagnostakis2023general}, the approximation scheme of~\cite{EtoLej} for classical diffusions with possibly discontinuous uniformly elliptic coefficients, and several simpler random-walk approximation frameworks, including the Donsker invariance principle~\cite{Don}, sticky random walks~\cite{Ami}, and oscillating random walks~\cite{vo2023afunctional}. In the metric-graph setting, our results also extend the invariance principle for Walsh Brownian motion established in~\cite{pavlyukevich2024walsh}.

In one dimension, STMCA complements EMCEL~\cite{ankirchner2022properties,AnkKruUru,ankirchner2021wasserstein} and continuous-time Markov chain methods~\cite{BouRabee2020,meier2021markovchain} for approximating diffusions governed by generalized second-order differential operators.
	
	\subsection*{Paper Outline}
	
	Section~\ref{sec_setting} establishes the framework for our analysis. Within this framework, we define the Space-Time Markov Chain Approximation (STMCA) and present our main convergence result (Theorem~\ref{thm_main}).
	
	Section~\ref{sec_transition} derives explicit formulas for the STMCA’s transition probabilities and times, with asymptotic estimates critical to later proofs. These results rely on a Green function and formula for metric graph diffusions (proofs in Appendices~\ref{app_dirichlet} and~\ref{app_Green}).

	Section~\ref{sec_embedding} establishes the embedding property of the STMCA and provides embedding time bounds for diffusions on metric graphs. 
	Section~\ref{sec_regularity} proves regularity estimates for general diffusions on metric graphs under a non-explosion condition that is expressed in terms of the speed measures of the diffusion. 
	Section~\ref{sec_proof} combines the results of the two previous sections to prove Theorem~\ref{thm_main} and the absorbed-boundary extension, Theorem~\ref{thm_absorbing}, and discusses adapted subdivisions.
	
Section~\ref{sec_numexp} illustrates the STMCA numerically
through three examples: Walsh Brownian motion with sticky
vertices, a star-graph diffusion with heterogeneous edge
dynamics, and a diffusion with an absorbing boundary.
	
	\section{Framework and main results}
	\label{sec_setting}

We introduce general diffusions on metric graphs,
define the STMCA on graph subdivisions,
and state the main convergence results.
	
	\subsection{General diffusions on metric graphs}
	\label{subsec_general}
	
	A finite metric graph is a graph $\Gamma = (V,E) $ with a finite number of edges to which we assign edge lengths: $e \mapsto l_e $, for all $e\in E $.
	
	One can construct such graphs by disjoint union of intervals modulo an equivalence relation.
	To be more precise, consider a finite family of intervals $\{U_i\}_{i\in\{1,2,\dots,N\}}$ so that each $U_i$ admits one of the following forms: $[0,l)$, $[0,l]$, or $[0,\infty)$, with $l>0$.
	Define the disjoint union of these intervals and the disjoint union of their closed endpoints:
	\begin{equation}
		\msf \Gamma \coloneqq \bigsqcup_{i \leq N} U_i \quad \text{and} \quad 
		\msf V \coloneqq \bigsqcup_{i \leq N} \cubraces{ U_i \cap \partial\, U_i }.
	\end{equation}
	Following~\cite{mugnolo2019actually}, given an equivalence relation on $\msf V$, the quotient set 
	$\Gamma \coloneqq \msf \Gamma/\sim$ defines a metric graph with vertices
	$V = \msf V /\sim$ and edges indexed by $\{1,\dots,N\}$ with edge-lengths $(l_i)_{i=1}^N$.
	Every element $\bx \in \Gamma$ can be represented as a pair of coordinates $\bx = (i, x)$, where $i \in \{1, \dots, N\}$ indicates the edge containing the point, and $x \in U_i$ specifies its location within that edge.
	Note that this representation is not unique at the vertices, where two different coordinates can refer to the same point in $\Gamma$.
	The graph structure of $\Gamma$ (described by the relation $\sim$) yields the equivalence:
	$(i,x) = (j,y)$ if and only if either $(i,x)=(j,y)$ or $(i,x),(j,y) \in \msf V$ with $(i,x) \sim (j,y)$.
	Throughout this paper we assume that the graph has \emph{no self-loops}: the two finite endpoints of an edge are never identified with the same vertex. Parallel edges between distinct vertices, and cycles consisting of several edges, are allowed. This convention ensures that each incidence at a vertex is represented by exactly one incident edge; treating self-loops would instead require indexing the transmission coefficients and star-neighborhood branches by oriented half-edges. 
	
We fix an orientation on $\Gamma$ as follows.
If an edge $e$ is incident to a vertex $\bv$, then
$e$ points inward to $\bv$ if $(e,l_e)=\bv$,
and outward from $\bv$ if $(e,0)=\bv$.
We denote by $E(\bv)$ the set of all edges incident to $\bv$,
by $E_{\cin}(\bv)$ the set of edges pointing inward to $\bv$,
and by $E_{\cout}(\bv)$ the set of edges pointing outward from $\bv$.
Figure~\ref{fig:edge_orientation} illustrates this convention.

\begin{figure}[htbp]
	\centering
	\includegraphics[width=.7\linewidth]{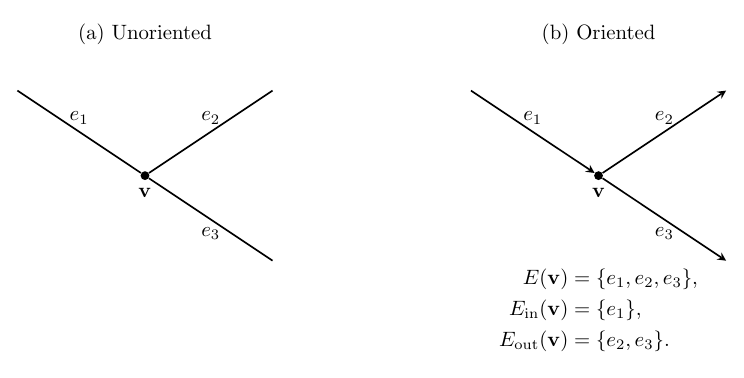}
	\caption{Edge orientation convention at a vertex $\mathbf v$.
		The unoriented graph (left) and an orientation of its
		incident edges (right).}
	\label{fig:edge_orientation}
\end{figure}
	
We equip $\Gamma$ with the geodesic distance $d$,
the induced topology, and the corresponding Borel
$\sigma$-algebra. The geodesic distance between two
points $\bx,\by\in\Gamma$ is the infimum of the lengths
of continuous paths joining them.\footnote{In the
	present finite metric-graph setting, this infimum
	is attained.}

We denote by $(\overline\Gamma,\overline d)$
the metric completion of $(\Gamma,d)$.
In particular, finite open endpoints of
$\Gamma$ belong to $\overline\Gamma$.

A basis for the topology is given by two types of
open sets: intervals contained in the interior of
a single edge, of the form
\[
Q=\{(i,\zeta):\zeta\in(a,b)\},
\qquad 0<a<b<l_i,
\]
and star-shaped neighborhoods of vertices:
\begin{equation}
	Q=\{\bv\}
	\cup\bigcup_{i\in E_{\cout}(\bv)}
	\{(i,\zeta):0<\zeta<u_i\}
	\cup\bigcup_{i\in E_{\cin}(\bv)}
	\{(i,\zeta):l_i-u_i<\zeta<l_i\},
\end{equation}
where $0<u_i<l_i$ for every $i\in E(\bv)$.
The Borel $\sigma$-algebra is generated by
these basic open sets.

	\medskip

A \emph{general diffusion} on $\Gamma$ is a regular
continuous strong Markov process, where continuity
is understood with respect to the geodesic distance $d$.

Here, regularity means that every point of $\Gamma$
is accessible with positive probability from every
interior point. More precisely, for all
$\bx\in\Gamma^\circ$ and $\by\in\Gamma$,
\[
\Prob_{\bx}(T_{\by}<\infty)>0,
\qquad
T_{\by}:=\inf\{t\ge0:\mX_t=\by\}.
\]
This definition does not impose accessibility from
boundary starting points, and therefore does not preclude absorption.
	
	\subsubsection{Analytical characterization}
	By \cite[Theorem~3.1]{freidlin1993diffusion}, a family of generalized second order differential operators $\msf L \coloneqq (\Lop_e)_{e\in E}$ (in the sense of~\cite{feller1955secondorder}) and a family of Wentzell transmission conditions $\msf G \coloneqq (\Gop_{\bv})_{\bv\in V} $ uniquely define a general diffusion process on $\Gamma $. 
	We assume the following forms for $\msf L $ and $\msf G $ which exclude	killing behavior (see \cite[Chapter~4]{ItoMcKean96} and \cite[Chapter~III]{Mandl1968} for the one-dimensional case):
	\begin{xenumerate}{c}
		\item \label{item_analytical1} For all $e\in E $, we have that $\Lop_e \coloneqq \frac{1}{2} \D_{m_e} \D_{s_e} $, where $m_e $ is a positive locally finite measure on $U_{e} $ and $s_e $ is a continuous strictly increasing function on $U_e $. The quantities $(s_e,m_e) $ play locally the role of scale and speed for one-dimensional diffusions. We will therefore call them \emph{scales} and \emph{speeds}. 
		
		\item \label{item_analytical2} For all $\bv\in V $, 
		\begin{equation}
			\label{eq_lateral}
			G_{\bv}(f) : \quad \sum_{j\in E_{\cout}(\bv)} \beta^{\bv}_{j} f'(j,0)
			- \sum_{j\in E_{\cin}(\bv)} \beta^{\bv}_{j} f'(j,l_j)
			= \rho_{\bv} \Lop f(\bv),
		\end{equation}
		where $\Lop f(\bv)$ denotes the common limit of $\Lop_e f$ along all edges incident to $\bv$, 
		where $\rho_{\bv} \ge 0$ and 
		\begin{equation}
			\sum_{j\in E_{\cout}(\bv)} \beta^{\bv}_{j}
			+ \sum_{j\in E_{\cin}(\bv)} \beta^{\bv}_{j} = 1, \quad 
			\beta^{\bv}_{e} >0, \quad \text{for all } e \in E(\bv).
		\end{equation}
	\end{xenumerate}
	
	The infinitesimal generator of the diffusion defined above is the operator
	defined as $\Lop f(e,x) = \Lop_e f(e,x) $, for all $(e,x)\in \Gamma $ and $f\in \dom(\Lop) $,
	where
	\begin{equation}
		\dom(\Lop) \coloneqq 
		\cubraces{f\in C_b(\Gamma):\; \Lop f \in C_b(\Gamma) \text{ and }
			\eqref{eq_lateral} \text{ holds for } f}
	\end{equation}
	
	Let us define a class of diffusions, first introduced in \cite{anagnostakis2025walsh}, that is central to our analysis.
	It is the class of diffusions on finite metric graphs that are locally on natural scale, on each edge.

\begin{definition}
	A diffusion defined by~\ref{item_analytical1}--\ref{item_analytical2}
	is said to be on \emph{Natural Scale on Edges} (NSE)
	if $s'_e\equiv 1$ for every $e\in E$.
	
	This notion should not be confused with a diffusion
	being on natural scale. Indeed, skew Brownian motion
	is not on natural scale in its usual representation
	on $\mathbb R$, but can be viewed as an NSE diffusion
	on a two-legged metric graph.
\end{definition}

\subsubsection{Probabilistic interpretation}

We now discuss the probabilistic interpretation
of~\ref{item_analytical1}--\ref{item_analytical2}.
Let $\mX=(I,X)$ be a general diffusion on
the finite metric graph $\Gamma$, defined on
the filtered probability space
$\mc P_{\bx}\coloneqq
(\Omega,\process{\bF_t},\Prob_{\bx})$,
with $\mX_0=\bx$, $\Prob_{\bx}$-almost surely. 
The scale functions and speed measures
$(s_e,m_e;\,e\in E)$ determine the local
hitting probabilities and occupation times.
Indeed, by localization, the characterization
of one-dimensional diffusions in terms of
scale and speed, and Dynkin's operator
(see~\cite[Theorem~17.23]{Kal},
\cite[Section~VII.3]{RevYor}, and
\cite[Proposition~3.6]{anagnostakis2025walsh}),
we have
\begin{align}
	\Prob_{(e,x)}
	\left(T_{(e,b)}<T_{(e,a)}\right)
	&=\frac{s_e(x)-s_e(a)}
	{s_e(b)-s_e(a)},\\
	\Esp_{(e,x)}
	\left[\int_0^{T_{(e,a)}\wedge T_{(e,b)}}
	h(X_t)\,\rd t\right]
	&=2\int_{(a,b)}
	G^e_{a,b}(x,y)h(y)\,m_e(\rd y),
\end{align}
for every $e\in E$, $0<a<x<b<l_e$,
and every nonnegative measurable function
$h:(a,b)\to\mathbb R$.
Here $G^e_{a,b}$ is the Green kernel
\begin{equation}
	\label{eq_Green}
	G^e_{a,b}(x,y)
	\coloneqq
	\frac{
		\bigl(s_e(x\wedge y)-s_e(a)\bigr)
		\bigl(s_e(b)-s_e(x\vee y)\bigr)
	}{
		s_e(b)-s_e(a)
	},
	\qquad x,y\in(a,b).
\end{equation}

The vertex transmission conditions in
\ref{item_analytical2} determine the spinning
probabilities and the asymptotic mean exit
times from small vertex neighborhoods.
More precisely, by
\cite[Proposition~3.7]{anagnostakis2025walsh}
and localization,
\begin{equation}
	\label{eq_prob_interpretation}
	\lim_{h\downarrow0}
	\Prob_{\bv}\bigl(I(T_{\bv,h})=e\bigr)
	=\beta_e^{\bv},
	\qquad
	\lim_{h\downarrow0}
	\frac{\Esp_{\bv}[T_{\bv,h}]}{h}
	=\rho_{\bv},
\end{equation}
for every $\bv\in V$ and $e\in E(\bv)$,
where
\[
T_{\bv,h}
\coloneqq
\inf\{t\ge0:d(\mX_t,\bv)>h\}.
\]

	\subsubsection{Boundary classification}
A diffusion on a finite metric graph may exhibit
different boundary behaviors at the endpoints of
its edges, namely at open endpoints or vertices
of degree one (incident to a single edge).
In the one-dimensional theory, such boundaries
are classified as natural, regular, exit, or entry.

By local equivalence in law with one-dimensional
diffusions on each edge
(see~\cite[Theorem~3.1]{freidlin1993diffusion})
and the classical Feller boundary classification
(see, e.g., \cite[Section~5.11]{Ito2006}),
consider a right-hand endpoint $x=l_e$ and
fix $c\in(0,l_e)$.
	The relevant Feller integrals, interpreted as increasing improper limits,
	are
	\begin{equation}
		\label{eq: boundary_classification}
		\begin{aligned}
		I_1&:=\lim_{b\uparrow l_e}\int_{(c,b)}m_e((c,y))\,s_e(\rd y),\\
		I_2&:=\lim_{b\uparrow l_e}\int_{(c,b)}
		\bigl(s_e(y)-s_e(c)\bigr)\,m_e(\rd y).
		\end{aligned}
	\end{equation}
	A left-hand endpoint is treated by reflecting the coordinate and applying
	the analogous tests to the reoriented scale and speed.
	If $I_1, I_2 < \infty $ the boundary is regular (i.e., the process can reach the boundary from $\Gamma^{\circ} $ and can reach $\Gamma^{\circ} $ from the boundary), if $I_1 < \infty $ and $I_2 = \infty $ the boundary is exit (i.e., the process can reach the boundary from $\Gamma^{\circ} $, but it cannot reach $\Gamma^{\circ} $ from the boundary), if $I_1 = \infty $ and $I_2 < \infty $ the boundary is entry (i.e., the process cannot reach the boundary from $\Gamma^{\circ} $, but it can reach $\Gamma^{\circ} $ from the boundary), and if $I_1 =  I_2 = \infty $ the boundary is natural (i.e., the process cannot reach the boundary from $\Gamma^{\circ} $ and cannot reach $\Gamma^{\circ} $ from the boundary).

In the original nonabsorbing setting, we work
with natural and regular boundaries. Any accessible
finite endpoint is included in $V$ and equipped
with its boundary condition. A finite open endpoint
outside $V$ is therefore understood to be
inaccessible (natural).

In the absorbing extension, we also allow
accessible finite endpoints at which the diffusion
is trapped upon its first visit. Such absorption
may arise intrinsically at an exit boundary,
as determined by the scale and speed measures,
or may be imposed at an otherwise regular boundary.

In this extension, every accessible finite open
endpoint is taken to be absorbing and is adjoined
to the state space through the metric completion,
as described in Definition~\ref{def_boundary_adapted}.
No transmission condition is imposed at absorbing
endpoints. The speed measures need only be locally
finite in the edge interiors and may have infinite
mass near an absorbing endpoint. The local operator
$\tfrac12\D_m\D_s$ is imposed only away from
the absorbing set $\mathcal A$.

\subsubsection{Star graphs and Walsh Brownian motion}

	We now introduce two objects of particular interest in this paper: star-graphs and the Walsh Brownian motion on finite metric graphs.
	
	To simplify the presentation, we express the
	scale functions and speed measures in local
	coordinates increasing away from each
	nonabsorbing original vertex $\bv$. 
	On an outgoing edge $e$, put
	$\widehat s^{\bv}_e(y)=s_e(y)-s_e(0)$ and let $\widehat m^{\bv}_e=m_e$.
	On an incoming edge $e$ of finite length $l_e$, assume the endpoint limit $s_e(l_e-)$ is finite and set
	\begin{equation}\label{eq_reoriented_scale}
	\widehat s^{\bv}_e(y)=s_e(l_e-)-s_e(l_e-y),\qquad
	\widehat m^{\bv}_e(B)=m_e(\{l_e-y:y\in B\}),\quad 0\le y<l_e,
	\end{equation}
	for Borel sets $B\subset(0,l_e)$. Both scales are strictly increasing, normalized to vanish at $\bv$, and the speed is the pushforward under reflection. In particular, for NSE diffusions $\widehat s^{\bv}_e(y)=y$ in either orientation. For the explicit non-NSE vertex formulas below we also require each outward scale to have a finite, strictly positive right derivative at the vertex; this condition is automatic for NSE diffusions.
	
Writing
\[
\widehat f_e^{\bv}(y)=
\begin{cases}
	f(e,y), & e\in E_{\cout}(\bv),\\
	f(e,l_e-y), & e\in E_{\cin}(\bv),
\end{cases}
\]
the transmission conditions~\ref{item_analytical2}
take the form
\begin{equation}
	\sum_{j\in E(\bv)}
	\beta_j^{\bv}(\widehat f_j^{\bv})'(0)
	=\rho_{\bv}\Lop f(\bv),
	\qquad \bv\in V.
\end{equation}
	
	\begin{definition}
		\label{def_stargraph}
		A \emph{star-graph} with $N$ edges is a graph consisting of a central vertex connected to $N$ outward edges. 
		Such a graph can be represented as 
		\begin{equation}
			\Gamma \coloneqq \braces{\bigsqcup_{i \leq N} [0,l_i)}/{\sim},
		\end{equation}
		where $l_i \in (0,\infty]$ is the length of the $i$-th edge for each $1 \leq i \leq N$, 
		and the equivalence relation $\sim$ identifies all points $(i,0)$ (the central vertex) 
		and is otherwise trivial. More precisely, for $\bx \coloneqq (i,x), \by \coloneqq (j,y) \in \Gamma$,
		we have $\bx \sim \by$ if and only if either $x = y = 0$ or $(i,x) = (j,y)$.
		The geodesic distance on $\Gamma$ is given by
		\[
		d\big((i, x), (j, y)\big) =
		\begin{cases}
			|x - y|, & \text{if } i = j, \\
			x + y,   & \text{if } i \neq j.
		\end{cases}
		\]
	\end{definition}
	
	Next, we define the Walsh Brownian motion. 
	In the conservative setting of this definition, Walsh Brownian motion is considered on graphs without finite open exit boundaries. The absorbing extension developed below concerns a different boundary convention. 
	Let $\Gamma$ be such a graph.
	
	\begin{definition}
		\label{def_walshBM}
		The \emph{Walsh Brownian motion} on $\Gamma$ is the general diffusion on $\Gamma $
		defined via (i) scales $s_e(x)=x $, for all $(e,x)\in \Gamma $, (ii) speeds $m_{e}(\rd x) = \vd x $, for all $(e,x)\in \Gamma $, and (iii) lateral conditions
		\begin{equation}
			\label{eq:lateral_Walsh}
			G_{\bv}(f) : \quad \sum_{e\in E_{\cout}(\bv)} \beta^{\bv}_{e} f'(e,0)
			- \sum_{e\in E_{\cin}(\bv)} \beta^{\bv}_{e} f'(e,l_e) = 0,
			\quad \bv\in V. 
		\end{equation}
	\end{definition}
	
	We observe that the Walsh Brownian motion on $\Gamma $ is fully characterized by its bias parameters $(\beta^{\bv}_{e})_{\bv,e} $
	Also it is an NSE general diffusion that behaves on each edge like a standard Brownian motion and locally on every vertex star-neighborhood like a Walsh Brownian motion on the star-graph.
	Since there is no second order term in~\eqref{eq:lateral_Walsh} and the speed measures are atomless, the process spends zero time on any vertex or other element $\bx $
	of $\Gamma $. 
	
	\subsection{Subdivisions of $\Gamma $}
	\label{subsec_subdivisions}
	
	The random walk approximation developed in this paper is defined on suitable subdivisions of the metric graph $\Gamma$. 
	In this section, we introduce graph subdivisions along with a thinness quantifier that measures their granularity. 
	This quantifier appears directly in our main convergence results  (Theorems~\ref{thm_main} and~\ref{thm_absorbing}), which bounds the $p$-Wasserstein distance between the laws of the diffusion and its STMCA approximation over a finite time horizon. 
	
	\begin{definition}
		A \emph{subdivision} of the metric graph
		$\Gamma=(V,E)$ is a graph
		$\Delta=(V_\Delta,E_\Delta)$ obtained by
		adding vertices to $V$ and splitting
		the edges of $\Gamma$ accordingly.
		We assume that $V_\Delta$ has no
		accumulation points in $\Gamma$.
		
		For each $\bx\in V_\Delta$, the
		\emph{cell} $U_{\bx}$ centered at
		$\bx$ is the connected component
		containing $\bx$ of the open set
		\[
		\Gamma\setminus
		(V_\Delta\setminus\{\bx\}).
		\]
		We denote by
		\[
		\mc C(\Delta)
		\coloneqq
		\{U_{\bx}:\bx\in V_\Delta\}
		\]
		the family of cells of $\Delta$.
	\end{definition}

\begin{definition}[Admissible subdivisions]
\label{def_admissible_subdivision}
A subdivision $\Delta=(V_\Delta,E_\Delta)$ is
\emph{admissible} if the following geometric conditions hold:
\begin{enumerate}
    \item Every finite original edge joining two vertices
    in $V$ contains at least one subdivision vertex in
    its interior.
    \item On every edge with a finite open endpoint,
    subdivision vertices approach that endpoint.
    \item On every infinite edge, the subdivision
    vertices have unbounded edge coordinates.
\end{enumerate}
Consequently, every cell centered at an original
vertex is a star-neighborhood with arm lengths
$0<u_e<l_e$, and every cell centered at an
additional subdivision vertex is a bounded edge
interval whose endpoints belong to $V_\Delta$.
\end{definition}

The admissibility requirement concerns the thinness
quantifiers and the resulting STMCA approximation
estimates. 
The stopped star-graph decomposition
in Section~\ref{sec_regularity} instead uses the
original vertex set $V$, corresponding to the
trivial subdivision $\Delta=\Gamma$, and does
not require admissibility.

	\begin{figure}[htbp]
		\centering
		\includegraphics[width=\linewidth]{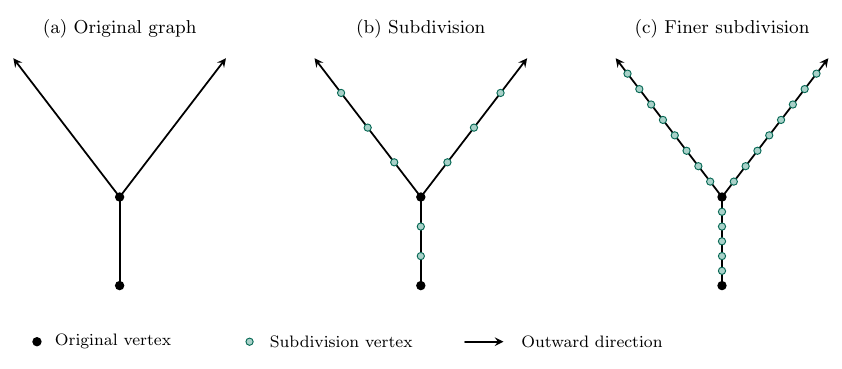}
\caption{A three-legged metric graph (a) and two
	successively finer subdivisions (b)--(c).
	Black dots denote original vertices, while
	green dots denote additional subdivision vertices.
	Arrows indicate outward directions along the
	two unbounded edges.}
\label{fig:graph_subdivision}
	\end{figure}
	
	For an admissible subdivision $\Delta$, the
thinness quantifier is defined through the following
\emph{measure of thinness} on cells of $\Delta$.
By admissibility, every cell is either an edge
interval or a vertex star-neighborhood.
	
	\begin{enumerate}
		\item 	Let $U$ be an edge-segment of $\Gamma $, i.e. 
		\begin{equation}
			\label{eq: vertex_neighborhood0}
			U \coloneqq \{(e,x) \mid x \in (a,b)\},
		\end{equation}
		for some $e\in E$ and finite $a,b$ with
		$0\le a<b\le l_e$. Equality at $0$ or at a
		finite $l_e$ is permitted only if that endpoint
		belongs to $\Gamma$ as an original vertex.
		Scale values there are interpreted by the
		corresponding endpoint limits.
		The thinness measure of $U$ is defined as: 
		\[
		\abs{U}_{\mX} \coloneqq m_e\big((a,b)\big) \cdot \abs{s_e(b) - s_e(a)},
		\]
		
		\item 	Let $U$ be a vertex star-neighborhood of $\Gamma $, i.e., a set of the form 
		\begin{equation}
			\label{eq: vertex_neighborhood}
			U \coloneqq \bigcup_{e \in E(\bv)} \{(e,x) \mid d(\mathbf{v},(e,x)) < u_e\},
		\end{equation}	
		where $\mathbf{v} \in V$ is a vertex of incident edges $E(\bv)$
		and $u_e < l_e $, for all $e\in E(\bv) $.
		The thinness measure of $U $ is defined as: 
\begin{align}
	\abs{U}_{\mX}
	\coloneqq {}&
	\rho_{\bv}
	\sum_{e\in E(\bv)}
	\frac{\widehat s_e^{\bv}(u_e)}
	{\beta_e^{\bv}}
	+
	\sum_{e\in E(\bv)}
	\widehat m_e^{\bv}((0,u_e))
	\widehat s_e^{\bv}(u_e).
\end{align}
		where $(\widehat s^{\bv}_e,\widehat m^{\bv}_e;\,
		e\in E(\bv),\,\bv\in V)$ are the reoriented
		scales and speeds introduced above. 
	\end{enumerate}
	
	\begin{definition}
		Let $\Gamma$ be a finite metric graph and
		$\Delta$ an admissible subdivision of $\Gamma$.
		The \emph{thinness quantifier} of $\Delta$
		is the quantity
		\[
		\abs{\Delta}_{\mX} \coloneqq \sup_{U \in \mathcal{C}(\Delta)} \abs{U}_{\mX},
		\]
		where $\mathcal{C}(\Delta)$ is the family of cells of $\Delta$. 
	\end{definition}

\begin{definition}[Classes of subdivisions]
	We define two classes of subdivisions:
	\begin{enumerate}
		\item A subdivision $\Delta$ is \emph{covering}
		if it is admissible and $|\Delta|_{\mX}<\infty$.
		In particular, admissibility requires subdivision
		vertices approaching every finite open natural
		endpoint and unboundedly far along every
		infinite edge.
		
		\item A subdivision $\Delta$ is
		\emph{$(\varepsilon,V)$-symmetric} if the ratios
		of edge-segment lengths incident to each original
		vertex are uniformly bounded by $1/\varepsilon$.
		More precisely, for every cell $U$ centered at
		$\bv\in V$ (represented as
		in~\eqref{eq: vertex_neighborhood}),
		the lengths $(u_e;\,e\in E(\bv))$ satisfy
		\[
		\varepsilon\leq\frac{u_i}{u_j}
		\leq\frac1\varepsilon,
		\qquad i,j\in E(\bv).
		\]
	\end{enumerate}
\end{definition}

\begin{definition}[Boundary-adapted subdivisions]
	\label{def_boundary_adapted}
	
Let $\mathcal A$ be a (possibly empty) finite set
of accessible finite edge endpoints at which the
diffusion is absorbed. If an endpoint in
$\mathcal A$ does not already belong to $\Gamma$,
we adjoin it as a vertex of the metric completion
$\overline\Gamma$. Throughout the absorbed
construction, $\Gamma$ and $V$ denote the
resulting enlarged graph and vertex set, with
the geodesic metric still denoted by $d$.
We consider admissible subdivisions of this enlarged
graph, with $\mathcal A\subseteq V_\Delta$.
	
	A cell $U_{\bx}$ is called \emph{active} if
	its center $\bx$ does not belong to $\mathcal A$.
	Cells centered at absorbing vertices are
	excluded, since the process remains trapped
	upon reaching $\mathcal A$.
	
	For an active edge cell $U=(a,b)$, with
	finite scale values at both endpoints,
	define its \emph{boundary-adapted thinness} by
	\begin{equation}
		\label{eq_boundary_thinness_cell}
		\Theta_X(U)
		\coloneqq
		2\int_{(a,b)}
		\frac{
			(s_e(y)-s_e(a))(s_e(b)-s_e(y))
		}{
			s_e(b)-s_e(a)
		}\,m_e(\rd y).
	\end{equation}
	Here $a$ or $b$ may be an absorbing endpoint,
	in which case the corresponding scale value
	is understood as a one-sided limit.
	The integral is interpreted improperly when
	the speed measure has infinite mass near
	that endpoint.
	
	For an active cell centered at an original
	vertex $\bv\in V\setminus\mathcal A$,
	we retain the original thinness $|U|_{\mX}$.
	
	The \emph{boundary-adapted thinness quantifier}
	of $\Delta$ is
	\begin{equation}
		\label{eq_boundary_thinness_mesh}
		h_\Delta^\partial
		\coloneqq |\Delta|_{\mX,\partial}
		=
		\sup_{\substack{U\in\mc C(\Delta)\\
				U\text{ active}}}
		\begin{cases}
			\Theta_X(U),
			& U\text{ is an edge cell},\\
			|U|_{\mX},
			& U\text{ is a vertex cell}.
		\end{cases}
	\end{equation}

	We call $\Delta$ \emph{boundary-adapted covering}
	if it is admissible, its active cells cover
	$\Gamma\setminus\mathcal A$,
	its subdivision vertices have no accumulation
	points in the accessible interior, and
	$h_\Delta^\partial<\infty$.
	\end{definition}
	
For the absorbing extension, the
$(\varepsilon,V\setminus\mathcal A)$-symmetry condition
is imposed only at nonabsorbing original vertices.

\begin{remark}
	\label{rmk_boundary_thinness_comparison}
	For an edge interval $(a,b)$ with finite speed mass,
	\begin{equation}
		\label{eq_boundary_old_comparison}
		\Theta_X(a,b)
		\le \frac12
		\bigl(s_e(b)-s_e(a)\bigr)m_e((a,b)).
	\end{equation}
	Consequently, whenever the original thinness is finite,
	\[
	h_\Delta^\partial\le |\Delta|_{\mX}.
	\]
	The boundary-adapted thinness therefore extends the
	original construction to cells where the speed measure
	may have infinite mass near an absorbing endpoint.
\end{remark}

	\subsection{The Space-Time Markov Chain Approximation} 
	\label{subsec_STMCA}

Let $\mX$ be a diffusion on a finite metric graph
$\Gamma$, defined on the filtered probability space
$\mc P_{\bx}\coloneqq
(\Omega,\process{\bF_t},\Prob_{\bx})$,
with $\mX_0=\bx$, $\Prob_{\bx}$-almost surely.
Let $\Delta=(\sV,\sE)$ be a covering subdivision,
or a boundary-adapted covering subdivision when
a set $\mathcal A$ of absorbing endpoints is specified.

For each $\bx\in\sV\setminus\mathcal A$,
let $U_{\bx}$ denote its cell and define
\[
\eta_{\bx}:=\inf\{t\ge0:\mX_t\notin U_{\bx}\}.
\]
For $\by\in\sV$, put
$T_{\by}:=\inf\{t\ge0:\mX_t=\by\}$.
At $a\in\mathcal A$, we impose the absorbing rule
$\fp_{a,a}=1$ and do not assign an outgoing holding time.

For each nonabsorbing vertex
$\bx\in\sV\setminus\mathcal A$, the transition
probabilities and conditional mean holding times are
\begin{equation}
	\fp_{\bx,\by}
	=\Prob_{\bx}(\eta_{\bx}=T_{\by}),
	\qquad
	\ft_{\bx,\by}
	=\Esp_{\bx}[\eta_{\bx}\mid\eta_{\bx}=T_{\by}],
	\quad \by\in\partial U_{\bx},
\end{equation}
whenever the conditioning event has positive probability;
otherwise $\ft_{\bx,\by}$ is undefined.

When the initial position $\bx$ does not belong to
$\sV$, let
\[
\sigma_0:=\inf\{t\ge0:\mX_t\in\sV\}
\]
and initialize the approximation from the hitting
distribution
\begin{equation}
	\mu_0(\by)
	=\Prob_{\bx}(\sigma_0=T_{\by}),
	\qquad \by\in\sV.
\end{equation}
If $\bx\in\sV$, set $\sigma_0=0$ and
$\mu_0=\delta_{\bx}$.

When $\bx\notin\sV$, the initial distribution
$\mu_0$ need not be a Dirac measure. This
initialization allows the STMCA to approximate
the diffusion from an arbitrary starting point
without requiring $\bx\in\sV$. The initial
hitting time is controlled separately in
Lemma~\ref{lem_initial_hit_bound}.

\begin{definition}[Space-Time Markov Chain Approximation]
		\label{def: STMCA}
		The \emph{Space-Time Markov Chain Approximation} (or STMCA) of $\mX$ on $\Delta$ is the $\sV$-valued random walk with initial distribution $\mu_0$, transition probabilities $(\fp_{\bx, \by})_{\bx,\by}$, and conditional transition times $(\ft_{\bx, \by})_{\bx,\by}$. More precisely, it is the process $\sX$
		\begin{itemize}
			\item whose consecutive values $(\mY_k)_{k\ge 0}$ form a discrete-time Markov chain on $V_\Delta$ with $\mY_0 \sim \mu_0$ and transition probabilities $(\fp_{\bx, \by})_{\bx,\by}$,
			
			\item whose consecutive jump times $(T_k)_{k\ge 0}$ are given by $T_0 = 0$ and
			\begin{equation}
				T_k - T_{k-1} = \ft_{\mY_{k-1}, \mY_k}, \quad 1\le k\le N\ \text{if } N<\infty,
			\end{equation}
		\end{itemize}
		With the above notations, the approximation process reads:
		\begin{equation}
			\sX_t=\mY_k \quad\text{when } T_k\le t<T_{k+1},\qquad t\ge0.
		\end{equation}
		If $\mY_N\in\mathcal A$ is the first absorbing vertex, the recursion stops at $N$,
		and $\sX_t=\mY_N$ for $t\ge T_N$; equivalently set $T_{N+1}=\infty$.
		If $N=\infty$, the holding-time recursion applies for all $k\ge1$.
	\end{definition}
	
	The STMCA has a Markov jump skeleton and transition-dependent holding times; in continuous time it is generally a semi-Markov process (the name STMCA refers to its Markov jump skeleton). Practical implementation requires computation of the transition quantities $(\fp_{\bx, \by})_{\bx,\by}$ and $(\ft_{\bx, \by})_{\bx,\by}$, which we discuss in Section~\ref{sec_transition}. 
	Note that the nontrivial transition quantities are needed only for neighboring vertices $\bx,\by\in\sV$, with $\bx\notin\mathcal A$. For regular interior and nonabsorbing vertex cells they are positive and finite; boundary-adapted cells have the same property by Lemma~\ref{lem_boundary_Green}.
	This process can be defined for a general diffusion on a finite metric graph whenever the resulting jump times are non-explosive.
	
	\subsection{Main result (convergence rate)}
	\label{subsec_main}
	
	We now state our main convergence results, which
	provide quantitative Wasserstein bounds between
	a diffusion on a finite metric graph and its STMCA.
	For $p\ge1$ and $T>0$, define the path-space
	$p$-Wasserstein distance by
	\begin{align}
		\label{eq_Wasserstein}
		\mc W^{p}_{T}(\law(\mX),\law(\mY)) &\coloneqq \mc W^{p}(\law(\mX_{t\in[0,T]}),\law(\mY_{t\in[0,T]})) 
		\\ &= \inf \left\{ \xnorm{\sup_{t\in [0,T]} d(\zeta_t,\xi_t)}{L^p(\Prob^{(\zeta,\xi)})} \mid (\zeta,\xi) \sim \Pi(\mX, \mY) \right\},
		\label{eq:Wasserstein}
	\end{align}
	where the infimum is taken over all possible couplings with marginals $\mX $ and $\mY $, we denote the family of couplings by $\Pi(\mX, \mY) $, and $\Prob^{(\zeta,\xi)} $ is the probability measure associated to the coupling $(\zeta,\xi) $.
	
	We impose the following condition on the diffusion $\mX$.
	
	\begin{condition}
		\label{cond_Lp}
		For every $p\ge1$ and $T>0$, there exists
		a constant $c>0$, independent of the starting
		point $\bx=(e,x)$, such that
		\begin{equation}
			\xnorm{d(\mX_s, \mX_t)}{L^{p}(\Prob_{\bx})} \le c (1+x)\sqrt{t-s},
			\quad \text{for all } 0\le s\le t\le T \text{ and } \bx\in\Gamma.
		\end{equation}
	\end{condition}
	
	While Condition~\ref{cond_Lp} suffices for our analysis, we also consider a stronger but more tractable requirement on the speed measures:
	
	\begin{condition}
		\label{cond_speed}
		There exists $k_1 >0 $ and $k_2 \ge 0 $
		such that
		\begin{equation}
			m_{e}(\rd y) \ge \frac{k_1}{1+k_2 y^{2}} \vd y,\quad 
			\text{for all } y \in U_e \text{ and } e\in E.
		\end{equation}
	\end{condition}

Condition~\ref{cond_speed} controls the growth of
the diffusion along each edge through a lower bound
on its speed measure. It is the finite metric graph
counterpart of the condition
$m(\rd y)\ge k_1\,\rd y/(1+k_2y^2)$,
under which convergence has been established for
EMCEL and one-dimensional STMCA.
Proposition~\ref{prop_Lp} shows that
Condition~\ref{cond_speed} implies
Condition~\ref{cond_Lp}.
	
	The following theorem gives our main quantitative
	convergence bound. Its proof is presented in
	Section~\ref{sec_proof}.
	
	\begin{theorem}
		\label{thm_main}
		Let $\mX$ be an NSE general diffusion on a finite metric graph $\Gamma$, defined on the probability space $(\Omega, \process{\bF_t},\Prob_{\bx})$ such that $\mX_0 = \bx = (e,x) $, $\Prob_{\bx} $-almost surely.
		Assume $\mX $ has only natural and regular boundaries.
		For every covering subdivision $\Delta $ of $\Gamma $, let $ \sX$ be the STMCA of $\mX $ on $\Delta $.
		Assume $\mX$ satisfies Condition~\ref{cond_Lp} or Condition~\ref{cond_speed}.  
		Then, for every fixed starting point $\bx$, $T>0 $, $p\ge 1 $, $\alpha \in (0, \frac{1}{4} \wedge \frac{1}{p}) $, and $\varepsilon>0 $, there exist constants $C>0$ and $h_0>0$, independent of $\Delta$, such that
		\begin{equation}
			\label{eq_Wasserstein_bound}
			\mc W^{p}_{T}(\law(\mX),\law(\sX)) \le C \xabs{\Delta}{\mX}^{\alpha},
		\end{equation}
		for all $(\varepsilon,V)$-symmetric covering subdivisions $\Delta$ of $\Gamma$ with $\xabs{\Delta}{\mX}\le h_0$.
	\end{theorem}
	
	\begin{remark}
		While we defined the STMCA for general diffusions on a metric graph $\Gamma$, the convergence result (Theorem~\ref{thm_main}) applies only to NSE diffusions. 
		However, many non-NSE cases can be reduced to NSE diffusions via transformations (see \cite[Section~7]{anagnostakis2025walsh} for star-graphs). 
		Under additional regularity (e.g., H\"older continuity of the scale functions  $(s_e)_{e\in E}$), analogous bounds to Theorem~\ref{thm_main} may be derived using arguments inspired by  \cite[Section~2.2]{anagnostakis2023general}. 
	\end{remark}
	
	\begin{remark}
		\label{rmk: absorption}
		An accessible exit endpoint is treated as an absorbing state of the
		completed graph, rather than being approached through infinitely many
		subdivision vertices. At such an endpoint the unweighted mass
		$m_e((a,b))$ may be infinite, so the original quantity
		$|s_e(b)-s_e(a)|m_e((a,b))$ cannot provide a useful bound.
		The weighted integral~\eqref{eq_boundary_thinness_cell} can remain
		finite and controls conditional durations leading \emph{towards}
		and \emph{away from} the endpoint. Absorption here is trapping
		at a graph point; killing to a cemetery state would require a
		corresponding metric on an enlarged state space.
	\end{remark}

\begin{theorem}[Absorbing-boundary extension]
	\label{thm_absorbing}
	Let $\mX$ be a non-explosive NSE diffusion on a finite
	metric graph with a finite set $\mathcal A$ of accessible
	absorbing endpoints, no interior killing, and only natural
	or regular nonabsorbing boundaries.
	Suppose its stopped-and-absorbed version satisfies
	Condition~\ref{cond_Lp} and is continuous in the
	completed geodesic metric.
	
	For every boundary-adapted covering subdivision $\Delta$,
	consider the absorbed STMCA from Definition~\ref{def: STMCA}.
	For every fixed starting point $\bx$, $T>0$, $p\ge1$,
	$0<\alpha<\frac14\wedge\frac1p$, and $\varepsilon>0$,
	there exist constants $C,h_0>0$, independent of $\Delta$,
	such that
	\begin{equation}\label{eq_Wasserstein_absorbing}
		\mc W_T^p(\law(\mX),\law(\sX))
		\le C\,(h_\Delta^\partial)^\alpha
	\end{equation}
	for all boundary-adapted covering subdivisions $\Delta$
	that are $(\varepsilon,V\setminus\mathcal A)$-symmetric
	and satisfy $h_\Delta^\partial\le h_0$.
\end{theorem}

	\begin{remark}
	\label{rmk_absorbed_regular}
	The extension assumes Condition~\ref{cond_Lp} directly; the sufficient
	Condition~\ref{cond_speed} is only proved to imply it in the original
	nonabsorbing setting. Theorem~\ref{thm_absorbing} does not make
	an assertion about killed paths in a cemetery space or about
	inaccessible endpoints with infinite scale.
	\end{remark}
	
	Since convergence in Wasserstein distance implies convergence in law (see \cite[Theorem~6.7]{villani2009optimal}), we have also the following. 
	
\begin{corollary}
	Fix $T>0$ and $\varepsilon>0$.
	Under the hypotheses of either
	Theorem~\ref{thm_main} or
	Theorem~\ref{thm_absorbing},
	let $(\Delta_n)_{n\ge1}$ be a sequence
	of subdivisions satisfying the corresponding
	covering and symmetry conditions.
	
	If $|\Delta_n|_{\mX}\to0$ in the
	nonabsorbing case, or
	$h_{\Delta_n}^{\partial}\to0$
	in the absorbing case, then the STMCA
	paths converge in law to the diffusion
	paths on $[0,T]$, with respect to the
	uniform metric on graph-valued paths.
\end{corollary}

	\section{On transition probabilities and transition times}
	\label{sec_transition}

	Let \(\mX\) be a general diffusion on the metric graph \(\Gamma\), and let \(\Delta\) be a subdivision of \(\Gamma\). 
	In this section, we establish three results concerning the transition probabilities \((\fp_{\bx,\by})_{\bx,\by}\) and times \((\ft_{\bx,\by})_{\bx,\by}\) of the STMCA: 
	\begin{enumerate}
		\item Explicit analytical representations of these quantities;
		\item Their asymptotic behavior near vertices; 
		\item Bounds in terms of the original thinness quantifier \(\xabs{\cdot}{\mX}\) and the boundary-adapted functional \(h_\Delta^\partial\).
	\end{enumerate}
	The analytical formulas enable direct computation
	of the STMCA transition quantities, while the
	asymptotic results characterize their behavior
	near vertices. The bounds are essential for
	proving Theorems~\ref{thm_main}
	and~\ref{thm_absorbing}.
	
	For an open cell $U$, write
	$T_U:=\inf\{t\ge0:\mX_t\notin U\}$ for its exit time.
	For each exit vertex $j\in\partial U$, define
	\begin{equation}\label{eq_def_vjk}
		v^j_k(z):=\Esp_z\bigl[T_U^k\mathbf1_{\{\mX_{T_U}=j\}}\bigr],
	\qquad z\in U,\quad k\in\IN_0.
	\end{equation}
	For a star neighborhood, we identify the exit vertex
	$\by_j=(j,u_j)$ with its incident edge index $j\in E(\bv)$.
	For an edge interval $(a,b)$ on an edge $e$, we instead use
	$j\in\{a,b\}$ to distinguish its two exit endpoints.

	Each cell of a subdivision is centered either at an original
	vertex $\bv\in V$ or at an interior point of an edge.
	For a cell $U_{\bx}$ centered at $\bx\in\sV\setminus\mathcal A$
	and an exit vertex $\by\in\partial U_{\bx}$,
	\[
	\fp_{\bx,\by}=v^{\by}_0(\bx),\qquad
	\ft_{\bx,\by}=\frac{v^{\by}_1(\bx)}{v^{\by}_0(\bx)},
	\]
	where the index $\by$ may be replaced by its edge label in a
	star neighborhood. Hence these two exit moments determine all
	allowed transition probabilities and conditional holding times.
	To simplify the presentation, we assume throughout this section that $\Gamma$ is a star-graph, except in Section~\ref{subsec_vjk_bound}, where the global character of the graph is needed. 
	Thus, $\Gamma$ is defined as in Definition~\ref{def_stargraph},
	we use outward coordinates with $s_e(0)=0$, and $\mX$ is
	determined by a family of scales and speeds
	$(s_e,m_e;\,e\in E)$ and a single junction vertex lateral condition:
	\begin{equation}
		\sum_{e' \in E} \beta_{e'} f'(e,0) = \rho \frac{1}{2} \D_{m_e} \D_{s_e} f(e,0),
		\quad \text{for all } e \in E.
	\end{equation}

	\subsection{The case: $U$ is an edge-interval}
	
	Assume $U$ is an edge-interval, that is a set of the form:
	\begin{equation}
		\label{eq_U_form2}
		U = \cubraces{(e,\zeta);\; \zeta\in (a,b) },
	\end{equation}
	with finite endpoints $0\le a<b\le l_e$
	that correspond to subdivision vertices. Equality
	at $0$ or $l_e$ is allowed only for original
	vertices, with scale values understood as endpoint
	limits. For accessible absorbing endpoints, the
	corresponding killed Green identities are justified
	by Lemma~\ref{lem_boundary_Green}.
	The following formulas hold for exits through the endpoints
	$j\in\{a,b\}$, with the endpoint index interpreted as above.
	
	\begin{proposition}
		\label{prop_moments1b}
		For every edge $e\in E$, finite endpoints
		$0\le a<x<b\le l_e$ corresponding to
		vertices of $\Delta$, with finite scale values
		at both endpoints, the exit probabilities from
		$(a,b)$ are
		\begin{equation}
			\Prob_{(e,x)}\braces{T_{(e,b)}<T_{(e,a)}}
			=\frac{s_e(x)-s_e(a)}{s_e(b)-s_e(a)},\qquad
			\Prob_{(e,x)}\braces{T_{(e,a)}<T_{(e,b)}}
			=\frac{s_e(b)-s_e(x)}{s_e(b)-s_e(a)}.
		\end{equation}
	\end{proposition}
	
	\begin{proof}
		The proof follows from the definition of the scale function of a one-dimensional (regular) diffusion (see, e.g., \cite[Section~VII.3]{RevYor}) and a localization argument.
	\end{proof}
	
	\begin{proposition}
		\label{prop_moments2b}
		For each exit endpoint $j\in\{a,b\}$, integer $k\ge1$,
		edge $e\in E$, and $x\in(a,b)$,
		\begin{equation}
			v^{j}_k(e,x)
			= 2k \int_{(a,b)} G^e_{a,b}(x,y) v^{j}_{k-1}(e,y)\, m_{e}(\rd y),
		\end{equation}
		where $(a,b,x,y)\mapsto G^e_{a,b}(x,y) $ is the Green kernel defined in~\eqref{eq_Green}.
	\end{proposition}
	
	\begin{proof}
		The proof follows from \cite[Proposition~3.2]{anagnostakis2023general}
		for one-dimensional diffusions and a localization argument. 
	\end{proof}

\begin{lemma}[Boundary Green bound and directed exit moments]
\label{lem_boundary_Green}
Let $U=(a,b)$ be an edge interval with finite endpoint scale values,
possibly having an endpoint in $\mathcal A$, and assume $\Theta_X(U)<\infty$.
Write $\tau_U$ for its exit time, $G=G^e_{a,b}$ for the killed
Green kernel (with endpoint values understood as limits), and
$p_j(z)=\Prob_z(\mX_{\tau_U}=j)$ for $j\in\{a,b\}$.

For $z\in(a,b)$, $j\in\{a,b\}$, and $k\ge1$,
\begin{align}
	\label{eq_directed_time}
	\ell_U(z\to j)
	&:=\Esp_z[\tau_U\mid\mX_{\tau_U}=j]
	=\frac{2}{p_j(z)}
	\int_{(a,b)}G(z,y)p_j(y)m_e(\rd y),\\
	\label{eq_directed_moment}
	\frac{\Esp_z[\tau_U^k\mathbf1_{\{\mX_{\tau_U}=j\}}]}
	{\Esp_z[\tau_U^{k-1}\mathbf1_{\{\mX_{\tau_U}=j\}}]}
	&\le k\Theta_X(U).
\end{align}
In particular, both directed mean exit times
are finite and
$\ell_U(z\to j)\le\Theta_X(U)$
for every $j\in\{a,b\}$.

\end{lemma}

\begin{proof}
The scale hitting formulas are
$p_b(z)=(s_e(z)-s_e(a))/(s_e(b)-s_e(a))$ and
$p_a=1-p_b$. Directly from the two cases $y\le z$ and $y\ge z$
in the Green kernel, for either $j=a,b$,
\begin{equation}\label{eq_Green_weight_domination}
G(z,y)\frac{p_j(y)}{p_j(z)}\le G(y,y),\qquad a<z,y<b.
\end{equation}
Define the positive killed potential operator
$Rf(z)=2\int_U G(z,y)f(y)m_e(\rd y)$.
The killed occupation formula and the strong Markov property give
$v^j_0=p_j$ and $v^j_k=kRv^j_{k-1}$.
For an accessible absorbing endpoint, these identities follow by
increasing compact subintervals and monotone convergence; the
finiteness of $\Theta_X(U)$ ensures $R1\le\Theta_X(U)$,
so the exit time is finite almost surely. From
\eqref{eq_Green_weight_domination},
$Rp_j\le\Theta_X(U)p_j$. Positivity and iteration imply
$R^kp_j\le\Theta_X(U)R^{k-1}p_j$. Since
$v^j_k=k!R^kp_j$, this proves the ratio bound and the
conditional-mean identity.
\end{proof}

	\subsection{The case: $U$ is a star-neighborhood}
	
	Assume $U$ is a star-neighborhood, that is a set of the form:
	\begin{equation}
		\label{eq_U_form1}
		U = \bv \cup \braces{\bigsqcup_{e\in E(\bv)} [0,u_e)}, \quad \text{with }
		u_e < l_e \text{ for all } e\in E(\bv).
	\end{equation}
	We now derive formulas for exits along each incident edge
	$j\in E(\bv)$.
	
	\begin{proposition}
		\label{prop_moments1}
		For every $j\in E(\bv)$ and $(e,x)\in U$,
		\begin{equation}
			v^j_0(e,x)
			=\frac{s_e(u_e)-s_e(x)}{s_e(u_e)}
			\frac{\beta^{\bv}_j s'_j(0)/s_j(u_j)}
			{\sum_{i\in E(\bv)}\beta^{\bv}_i s'_i(0)/s_i(u_i)}
			+\indic{j=e}\frac{s_e(x)}{s_e(u_e)}.
		\end{equation}
	\end{proposition}

	\begin{proof}
Write $q_j=v^j_0(\bv)$ for the probability
of exiting $U$ along edge $j$ when starting
from $\bv$, and set
\[
H_e(x)=\frac{s_e(u_e)-s_e(x)}{s_e(u_e)}.
\]
By Proposition~\ref{prop_moments1b} and the
strong Markov property at the first hitting
time of $\bv$ or $(e,u_e)$, we obtain
\begin{equation}\label{eq_v0_Markov}
	v^j_0(e,x)=H_e(x)q_j
	+\indic{e=j}\frac{s_e(x)}{s_e(u_e)}.
\end{equation}
		For $0<h<\min_{i\in E(\bv)}u_i$, the strong Markov
		property at $T_{\bv,h}$ yields
		\begin{equation}
			q_j=\sum_{i\in E(\bv)}
			\Prob_{\bv}\{I(T_{\bv,h})=i\}\,v^j_0(i,h).
		\end{equation}
		Subtract $q_j$, divide by $h$, and let $h\downarrow0$.
		Using~\eqref{eq_prob_interpretation} and the positive finite
		right derivatives $s'_i(0)$, we obtain
		\begin{equation}
			0=\sum_{i\in E(\bv)}\beta^{\bv}_i
			(v^j_0)'(i,0+)
			=\sum_{i\in E(\bv)}
			\frac{\beta^{\bv}_i s'_i(0)}{s_i(u_i)}
			(\indic{i=j}-q_j).
		\end{equation}
		Consequently,
		\[
		q_j=\frac{\beta^{\bv}_j s'_j(0)/s_j(u_j)}
		{\sum_{i\in E(\bv)}\beta^{\bv}_i s'_i(0)/s_i(u_i)},
		\]
		and substitution into~\eqref{eq_v0_Markov} proves the claim.
	\end{proof}

		\begin{proposition}
		\label{prop_moments2}
		For $j\in E$, $k\ge1$, and $(e,x)\in U$, write
		$H_e(x)=(s_e(u_e)-s_e(x))/s_e(u_e)$ and
		$d_i=\beta^{\bv}_i s'_i(0)/s_i(u_i)$, $D=\sum_{i\in E}d_i$.
		Then, 
		\begin{align}
		v^j_k(e,x)=k\bigg[&2\int_{(0,u_e)}
		G^e_{0,u_e}(x,y)v^j_{k-1}(e,y)\,m_e(\rd y)\\
		&+\frac{H_e(x)}{D}\bigg(\rho_{\bv}v^j_{k-1}(\bv)
		+2\sum_{i\in E}d_i\int_{(0,u_i)}
		(s_i(u_i)-s_i(y))v^j_{k-1}(i,y)\,m_i(\rd y)\bigg)\bigg].
		\end{align}
	\end{proposition}

\begin{proof}
	By Lemma~\ref{lem_vk_first_representation},
	\[
	\frac{v^j_k(z)}{k}
	=\Esp_z\left[
	\int_0^{T_U}v^j_{k-1}(\mX_t)\,\rd t
	\right].
	\]
	The function $v^j_0$ is bounded and continuous
	by Proposition~\ref{prop_moments1}; the same
	properties hold inductively for $v^j_{k-1}$
	by the Green representation.
	
	Applying Proposition~\ref{prop_Green} and
	the Dirichlet representation
	\eqref{eq_solution_Dirichlet}--\eqref{eq_constants_Dirichlet}
	with source term $g=-v^j_{k-1}$ and
	zero boundary data gives the stated formula.
\end{proof}
	
\begin{lemma}
\label{lem_vk_first_representation}
For an open cell $U$ and an exit vertex $j\in\partial U$,
let $v^j_k$ be as in~\eqref{eq_def_vjk}. Then, for
$k\ge1$ and $z\in U$,
\begin{equation}
 v^j_k(z)=k\Esp_z\left[\int_0^{T_U}
 v^j_{k-1}(\mX_t)\,\rd t\right],
\end{equation}
whenever these quantities are finite.
\end{lemma}

\begin{proof}
	Since $T_U<\infty$ almost surely,
	\[
	T_U^k=k\int_0^{T_U}(T_U-t)^{k-1}\,\rd t.
	\]
	By Tonelli's theorem and the strong Markov
	property at time $t$,
	\[
	\begin{aligned}
		&\Esp_z\left[
		\mathbf1_{\{t<T_U\}}
		\mathbf1_{\{\mX_{T_U}=j\}}
		(T_U-t)^{k-1}
		\mid\mathcal F_t
		\right]\\
		&\qquad=
		\mathbf1_{\{t<T_U\}}
		v^j_{k-1}(\mX_t).
	\end{aligned}
	\]
	Integrating over $t\ge0$ gives the assertion.
\end{proof}

	\subsection{Asymptotic behavior of $v^{j}_1$}
	
For every $\bv\in V$ and sufficiently small $h>0$, let
\[
U_h\coloneqq\{\bx\in\Gamma:d(\bx,\bv)<h\},
\qquad
T_{U_h}\coloneqq
\inf\{t\ge0:\mX_t\notin U_h\},
\]
and define
\[
v_k^j(\bv;h)
\coloneqq
\Esp_{\bv}\left[
\indic{I(T_{U_h})=j}T_{U_h}^k
\right],
\qquad k\in\IN_0.
\]

We recall the asymptotic~\eqref{eq_prob_interpretation}
for $v_0^j(\bv;h)$.
We now establish the corresponding asymptotic
for $v_1^j(\bv;h)$, which characterizes
the small-neighborhood behavior of the
destination-conditional mean exit times.

These asymptotic approximations are not used
in our numerical experiments, where we instead
evaluate the exact integral representations numerically.
	
		\begin{proposition}
		\label{prop_asymptotic}
		Assume the diffusion is NSE and $\rho_{\bv}>0$. For every $j\in E(\bv)$,
		\begin{equation}
		\lim_{h\downarrow0}\frac{v^j_1(\bv;h)}h
		=\rho_{\bv}\beta^{\bv}_j.
		\end{equation}
	\end{proposition}

	\begin{proof}
		For symmetric star neighborhoods of radius $h$ and on natural scale,
		Proposition~\ref{prop_moments1} gives $v^j_0(\bv;h)=\beta^{\bv}_j$.
		In Proposition~\ref{prop_moments2}, set $x=0$, $u_i=h$,
		$D=1/h$, and $d_i=\beta^{\bv}_i/h$. Then
		\begin{equation}
		v^j_1(\bv;h)
		=\rho_{\bv}\beta^{\bv}_j h
		+2\sum_{i\in E(\bv)}\beta^{\bv}_i
		\int_{(0,h)}(h-y)v^j_0(i,y)\,m_i(\rd y).
		\end{equation}
		The remainder is at most
		$2h\sum_i\beta^{\bv}_i m_i((0,h))=o(h)$
		by local finiteness of the speed measures at the vertex
		(the vertex atom is represented separately by $\rho_{\bv}$).
	\end{proof}

	\subsection{Useful bounds}
	\label{subsec_vjk_bound}
	
	For a covering or boundary-adapted covering subdivision $\Delta$ of $\Gamma$, if $U$ is a cell centered at $\bv\in V\setminus\mathcal A$
	defined as~\eqref{eq_U_form1}, let:
	\begin{equation}
		c_{\bv,\Delta} \coloneqq \braces{\sum_{i\in E} \frac{\beta^{\bv}_i}{u_i}}^{-1}  \min_{i\in E} \frac{\beta^{\bv}_i}{u_i} \quad \text{and} \quad
		c_{\Delta} \coloneqq \min_{\bv \in V\setminus\mathcal A} c_{\bv,\Delta},\qquad\text{with }c_\Delta:=1\text{ if }V\setminus \mathcal A=\emptyset.
	\end{equation}
	
		\begin{proposition}
		\label{prop_moments3}
		Let $\mX$ be an NSE diffusion, possibly absorbed at the finite
		endpoint set $\mathcal A$, and let $\Delta$ be boundary-adapted covering
		(with $\mathcal A=\emptyset$ in the original case).
		Set $c_{\Delta}$ as above and $K_{\Delta}=5/c_{\Delta}$.
		For $k\ge1$ and every allowed transition from
		$\bx\in\sV\setminus\mathcal A$ to an exit vertex $j$,
		\begin{equation}\label{eq_bound}
		\frac{v^j_k(\bx)}{v^j_{k-1}(\bx)}
		\le kK_{\Delta}h_\Delta^\partial.
		\end{equation}
		For a $(\varepsilon,V\setminus\mathcal A)$-symmetric subdivision,
		$c_{\Delta}\ge\varepsilon\beta_*>0$, where
		$\beta_*:=\min_{\bv\in V\setminus\mathcal A,e\in E(\bv)}\beta_e^{\bv}$;
		if $V\setminus\mathcal A$ is empty take $\beta_*=1$.
		For an original covering subdivision with $\mathcal A=\emptyset$, the
		bound also holds with $h_\Delta^\partial$ replaced by
		$|\Delta|_{\mX}$.
	\end{proposition}

\begin{proof}
	First consider a star cell $U$ centered at $\bv$.
	Write $h_U=|U|_{\mX}$ and set
	\[
	D=\sum_i\frac{\beta_i}{u_i},
	\qquad
	p_j=\frac{\beta_j/u_j}{D},
	\qquad
	H_e(x)=\frac{u_e-x}{u_e}.
	\]
	By Proposition~\ref{prop_moments1},
	\[
	0\le v^j_0(e,x)\le1,\qquad
	v^j_0(\bv)=p_j,\qquad
	v^j_0(e,x)\ge p_jH_e(x).
	\]
	
	Let $R$ denote the positive killed occupation
	operator on $U$, defined by
	\[
	Rf(z)=\Esp_z\left[
	\int_0^{T_U}f(\mX_t)\,\rd t
	\right].
	\]
	By Proposition~\ref{prop_moments2},
	$v_k^j=kRv_{k-1}^j$.
	
	Since $G^e_{0,u_e}(x,y)\le u_eH_e(x)$,
	\begin{equation}
		\frac{
			2\int_{(0,u_e)}
			G^e_{0,u_e}(x,y)v^j_0(e,y)m_e(\rd y)
		}{
			v^j_0(e,x)
		}
		\le
		\frac{2u_em_e((0,u_e))}{p_j}
		\le\frac{2h_U}{c_\Delta}.
	\end{equation}
	
	For the common vertex term of $Rv^j_0$,
	use $v^j_0\le1$,
	\[
	\frac{\rho_{\bv}}{D}
	\le\rho_{\bv}\sum_i\frac{u_i}{\beta_i}
	\le h_U,
	\qquad
	u_im_i((0,u_i))\le h_U.
	\]
	Thus
	\begin{align}
		&\frac{H_e(x)}{Dv^j_0(e,x)}
		\left(
		\rho_{\bv}p_j
		+2\sum_i\frac{\beta_i}{u_i}
		\int_{(0,u_i)}
		(u_i-y)v^j_0(i,y)m_i(\rd y)
		\right)\\
		&\qquad\le
		\frac{\rho_{\bv}}{D}
		+\frac{2}{p_j}
		\max_i u_im_i((0,u_i))
		\le h_U+\frac{2h_U}{c_\Delta}.
	\end{align}
	
	Consequently,
	\[
	Rv^j_0
	\le\left(1+\frac4{c_\Delta}\right)h_Uv^j_0
	\le K_\Delta h_Uv^j_0,
	\]
	where the last inequality follows from
	$c_\Delta\le1$ and $K_\Delta=5/c_\Delta$.
	By positivity of $R$, iteration gives
	\begin{equation}
		v^j_k=k!R^kv^j_0
		\le kK_\Delta h_U\,
		(k-1)!R^{k-1}v^j_0
		=kK_\Delta h_Uv^j_{k-1}.
	\end{equation}
	
	For an edge cell, including one adjacent
	to an absorbing endpoint,
	Lemma~\ref{lem_boundary_Green} gives
	\[
	\frac{v_k^j}{v_{k-1}^j}
	\le k\Theta_X(U)
	\le kh_\Delta^\partial.
	\]
	Since $K_\Delta\ge1$ and
	$h_U\le h_\Delta^\partial$ for star cells,
	the asserted bound follows in both cases.
	
	Finally, symmetry at the nonabsorbing
	original vertices gives
	\begin{equation}
		p_j=
		\frac{\beta_j/u_j}
		{\sum_i\beta_i/u_i}
		\ge\varepsilon\beta_j,
	\end{equation}
	because $\sum_i\beta_i=1$.
	Since there are finitely many original
	vertices, it follows that
	$c_\Delta\ge\varepsilon\beta_*>0$.
	
	In the nonabsorbing case,
	$h_\Delta^\partial\le|\Delta|_{\mX}$
	by~\eqref{eq_boundary_old_comparison},
	which yields the final assertion.
\end{proof}

	\begin{corollary}
		\label{cor_aggregate_moments_bound}
		For $k\ge1$ and all allowed nonabsorbing transitions,
		\begin{equation}
		v_k^j(\bx)\le k!\bigl(K_\Delta h_\Delta^\partial\bigr)^k.
		\end{equation}
	\end{corollary}

	\section{Embedding and properties}
	\label{sec_embedding}

We establish the embedding property of the STMCA
and derive bounds on the associated embedding times.
Throughout this section, we consider either the
nonabsorbing diffusion of Theorem~\ref{thm_main}
or the absorbed diffusion of
Theorem~\ref{thm_absorbing}.
	
	Write $\tau_{\mathcal A}=\inf\{t\ge0:\mX_t\in\mathcal A\}$ (with $\tau_{\mathcal A}=\infty$
	when $\mathcal A=\emptyset$). Let $\Delta$ be a covering or a
	boundary-adapted covering subdivision, respectively. Set
	\begin{align}
	\tau^\Delta_0&:=\inf\{t\ge0:\mX_t\in\sV\},\\
	\tau^\Delta_{k+1}&:=\inf\{t>\tau^\Delta_k:
	\mX_t\in\sV\setminus\{\mX_{\tau^\Delta_k}\}\},
	\end{align}
	until the first index $N$ with $Y_N:=\mX_{\tau_N}\in\mathcal A$;
	put $N=\infty$ if absorption does not occur. For bookkeeping
	on $\{N<\infty\}$, set $Y_k=Y_N$ and $\tau_k=\tau_N$ for
	$k>N$. These frozen values do not represent new jumps.
	Define the skeleton-conditioned clock by
	\begin{equation}\label{eq_clock_definition}
	\begin{aligned}
	 a_k&=\Esp_{\bx}[\tau_k-\tau_{k-1}\mid Y_{k-1},Y_k],
	 \quad 1\le k\le N,\\
	 S_0&=0,\qquad S_k=\sum_{j=1}^k a_j \quad(k\le N),\\
	 S_k&=\infty\quad(k>N),\qquad
	 K^\Delta(t)=\max\{0\le k\le N:S_k\le t\}.
	\end{aligned}
	\end{equation}
	With the convention $a_k=0$ for $k>N$, the quantities $\delta_k
	=\tau_k-\tau_{k-1}$ and $M_k=\sum_{j=1}^k(\delta_j-a_j)$
	remain well-defined. Before absorption the embedded process is
	$Y_{K^\Delta(t)}$; after $S_N$ it stays at $Y_N$.
	For a non-explosive diffusion, embedding times do not accumulate
	before reaching $\mathcal A$. On $\{N=\infty\}$, conditional independence
	given the skeleton (Lemma~\ref{lem_conditioning}) implies
	$S_k\to\infty$ almost surely: otherwise conditional expectation
	of the total actual holding time would be finite on an event of
	positive probability, contradicting non-explosion. Thus $K^\Delta(t)$
	is finite for every finite $t$, also when $N<\infty$.

	\begin{proposition}
		\label{prop_embedding}
		With the above notation,
		$\bigl(\mX_{\tau^{\Delta}_{K^{\Delta}(t)}}\bigr)_{t\ge0}\eqlaw\sX$.
	\end{proposition}

\begin{proof}
	By the strong Markov property at $\tau_0$ and
	at successive embedding times, the skeleton
	$(Y_k)_{k\ge0}$ has initial distribution $\mu_0$
	and transition probabilities $(\fp_{\by,\bz})$.
	Its law is therefore that of the STMCA skeleton.
	
	Moreover, for every transition before absorption,
	\[
	a_k=\Esp_\bx[\tau_k-\tau_{k-1}\mid Y_{k-1},Y_k]
	=\ft_{Y_{k-1},Y_k}.
	\]
	Thus the clock $(S_k)$ is determined by the
	skeleton through the same holding times as the
	STMCA. With the absorbing convention when
	$N<\infty$, the two right-continuous step
	processes have the same law.
\end{proof}

	We now state bounds on embedding times, extending the argument of
	\cite[Theorem~4.7 and Lemma~4.4]{anagnostakis2023general} to the
	boundary-adapted mesh. Throughout the next two propositions put
	$h=h_\Delta^\partial$, $a=K_\Delta h$ and
	$K_\Delta=5/c_\Delta$ (Proposition~\ref{prop_moments3}); on a
	traditional subdivision with no absorbing endpoints the estimate
	$h\le|\Delta|_{\mX}$ recovers the original bound.
	
	Write $\mc B=\sigma(Y_k:k\ge0)$ for the sigma-algebra
	generated by the full vertex skeleton, including its
	constant absorbing tail when $N<\infty$.

	\begin{proposition}
		\label{prop_embedding_bounds}
		Let $h=h_\Delta^\partial$ and $a=K_\Delta h$.
		For all $T>0$,
		\begin{equation}\label{eq_clock_L2}
		\Esp_\bx\left[\sup_{0\le t\le T}
		\big|(t\wedge\tau_{\mathcal A})-\tau^\Delta_{K^\Delta(t)}\big|^2\right]
		\le24aT+9a^2.
		\end{equation}
	\end{proposition}

	\begin{proposition}
		\label{prop_probability_bound}
		For $t,M>0$ and $0<\lambda a<1$, where $a=K_\Delta h_\Delta^\partial$,
		\begin{equation}
		\label{eq_embedding_tail}
		\Prob_\bx\{\tau^\Delta_{K^\Delta(t)}>M\}
		\le\exp\left\{\lambda\left(\frac{t+a}{1-\lambda a}-M\right)\right\}.
		\end{equation}
	\end{proposition}

	\subsection{Proof of Proposition~\ref{prop_embedding_bounds}}

	We first record the conditional independence of the embedded excursions.
	\begin{lemma}
		\label{lem_ET_martingale}
		Let $\mc B=\sigma(Y_k:k\ge0)$ and
		$\delta_k=\tau_k-\tau_{k-1}$ for $k\ge1$ (zero for $k>N$).
		Conditional on $\mc B$, the pre-absorption durations
		$(\delta_k)_{1\le k\le N}$ are independent with means
		$a_k=\ft_{Y_{k-1},Y_k}$.
		With
		\begin{equation}
		M_n=\sum_{k=1}^n(\delta_k-a_k),\qquad
		\mc G_n=\mc B\vee\sigma(\tau_0,\delta_1,\ldots,\delta_n),
		\end{equation}
		$(M_n)$ is a square-integrable $(\mc G_n)$-martingale
		(after localization where required); it is constant for $n\ge N$.
	\end{lemma}
	\begin{proof}
		Before absorption, repeated strong Markov arguments give the
		product conditional law of the excursions up to the first visit
		to $\mathcal A$; see Lemma~\ref{lem_conditioning}. Beyond $N$ the
		durations and centered increments are identically zero. Hence
		$\Esp[\delta_k-a_k\mid\mc G_{k-1}]=0$ in either case,
		and square integrability follows from
		Proposition~\ref{prop_moments3} after stopping at finite indices.
	\end{proof}

	\begin{lemma}
		\label{lem_ET_variance_bound}
		Let $a=K_\Delta h_\Delta^\partial$. Almost surely,
		\begin{equation}
		\label{eq_variance_sum}
		\sum_{k=1}^{K^\Delta(T)}
		\Var(\delta_k\mid Y_{k-1},Y_k)\le2aT.
		\end{equation}
	\end{lemma}

	\begin{proof}
		Proposition~\ref{prop_moments3} gives
		\begin{equation}
		\Var(\delta_k\mid Y_{k-1},Y_k)
		\le\Esp[\delta_k^2\mid Y_{k-1},Y_k]
		\le2a\,\Esp[\delta_k\mid Y_{k-1},Y_k]=2aa_k.
		\end{equation}
		Sum and use $S_{K^\Delta(T)}\le T$.
	\end{proof}

	\begin{lemma}
		\label{lem_initial_hit_bound}
		Let $a=K_\Delta h_\Delta^\partial$ and $\tau_0$ be the
		first hitting time of $\sV$. Then
		\begin{equation}
		\Esp[\tau_0\mid\mc B]\le a,\qquad
		\Esp[\tau_0^2\mid\mc B]\le2a^2.
		\end{equation}
		Both sides vanish when $\bx\in\sV$.
	\end{lemma}

\begin{proof}
	If $\bx\in\sV$, then $\tau_0=0$.
	Otherwise, $\bx$ lies in an interval
	$I=(a_0,b_0)$ between consecutive
	subdivision vertices, and $\tau_0$
	is its exit time.
	
	By admissibility, at least one endpoint
	of $I$ is an additional subdivision vertex.
	Its active edge cell contains $I$.
	Monotonicity of the diagonal killed Green
	kernel under interval inclusion therefore gives
	\[
	\Theta_X(I)\le h_\Delta^\partial.
	\]
	
	Lemma~\ref{lem_boundary_Green} therefore
	yields, for each possible exit endpoint
	$j\in\{a_0,b_0\}$,
	\[
	\Esp_\bx[\tau_0\mid Y_0=j]
	\le h_\Delta^\partial
	\le a,
	\]
	and
	\[
	\Esp_\bx[\tau_0^2\mid Y_0=j]
	\le 2(h_\Delta^\partial)^2
	\le 2a^2.
	\]
	Finally, by the strong Markov property
	at $\tau_0$, conditioning on the subsequent
	skeleton adds no information about $\tau_0$
	beyond $Y_0$. Hence
	\[
	\Esp[\tau_0\mid\mc B]\le a,
	\qquad
	\Esp[\tau_0^2\mid\mc B]\le2a^2.
	\]
\end{proof}

	\begin{proof}[Proof of Proposition~\ref{prop_embedding_bounds}]
		Conditional on $\mc B$, $K^\Delta(T)$ is deterministic and
		never exceeds $N$. Doob's $L^2$ maximal inequality and
		Lemma~\ref{lem_ET_variance_bound} give
		\begin{equation}
		\Esp\left[\max_{0\le k\le K^\Delta(T)}|M_k|^2
		\,\middle|\,\mc B\right]
		\le4\sum_{k=1}^{K^\Delta(T)}
		\Var(\delta_k\mid Y_{k-1},Y_k)\le8aT.
		\end{equation}
		Put $D_T=\sup_{t\le T}|(t\wedge\tau_{\mathcal A})-\tau_{K^\Delta(t)}|$.
		For $K^\Delta(t)<N$, the clock overshoot is at most $a$;
		if $t\ge\tau_{\mathcal A}$ then $\tau_{K^\Delta(t)}\le\tau_{\mathcal A}\le t$,
		so replacing $t$ by $t\wedge\tau_{\mathcal A}$ can only decrease the
		error. If $K^\Delta(t)=N$, then $t\ge S_N$ and
		\begin{equation}
		|(t\wedge\tau_{\mathcal A})-\tau_N|
		=(\tau_N-t)_+\le(\tau_N-S_N)_+
		\le\tau_0+|M_N|.
		\end{equation}
		In both cases,
		\begin{equation}\label{eq_stopped_clock_pathwise}
		D_T\le\tau_0+\max_{k\le K^\Delta(T)}|M_k|+a.
		\end{equation}
		Use $(u+v+w)^2\le3(u^2+v^2+w^2)$ and
		Lemma~\ref{lem_initial_hit_bound} to obtain
		$\Esp[D_T^2]\le3(2a^2+8aT+a^2)=24aT+9a^2$.
	\end{proof}

	\subsection{Proof of Proposition~\ref{prop_probability_bound}}

\begin{lemma}
	\label{lem_exponential1}
	Let $H$ be either an embedded duration
	$\delta_k = \tau_k-\tau_{k-1}$, $1\le k\le N$, or the initial
	duration $\tau_0$. Write $\mathcal C$ for
	the conditioning on $(Y_{k-1},Y_k)$ in
	the first case and on $Y_0$ in the second.
	With $a=K_\Delta h_\Delta^\partial$
	and $0<\lambda a<1$,
	\begin{equation}
		\label{eq_local_mgf}
		\Esp[e^{\lambda H}\mid\mathcal C]
		\le
		\exp\left\{
		\frac{\lambda\Esp[H\mid\mathcal C]}
		{1-\lambda a}
		\right\}.
	\end{equation}
\end{lemma}

\begin{proof}
	By Proposition~\ref{prop_moments3},
	and by Lemma~\ref{lem_boundary_Green}
	for the initial edge interval,
	\[
	\Esp[H^k\mid\mathcal C]
	\le k!a^{k-1}\Esp[H\mid\mathcal C],
	\qquad k\ge1.
	\]
	Consequently, 
	\begin{align}
		\Esp[e^{\lambda H}\mid\mathcal C]
		&\le
		1+\frac{\lambda}{1-\lambda a}
		\Esp[H\mid\mathcal C]\\
		&\le
		\exp\left\{
		\frac{\lambda\Esp[H\mid\mathcal C]}
		{1-\lambda a}
		\right\}.
	\end{align}
	The first inequality follows by expanding the
	exponential into its power series and summing
	the resulting geometric series. The second
	uses $1+x\le e^x$ for $x\ge0$.
\end{proof}

\begin{proof}[Proof of Proposition~\ref{prop_probability_bound}]
	Condition on the complete skeleton $\mc B$.
	By Lemma~\ref{lem_conditioning}, the initial
	hitting time $\tau_0$ and the inter-embedding
	times $(\delta_k)_{1\le k\le N}$ are
	conditionally independent. Moreover,
	$K^\Delta(t)\le N$ is $\mc B$-measurable.
	
	Since
	\[
	\tau_{K^\Delta(t)}
	=\tau_0+\sum_{k=1}^{K^\Delta(t)}\delta_k,
	\]
	Lemma~\ref{lem_exponential1} gives
	\begin{align}
		\Esp_\bx[e^{\lambda\tau_{K^\Delta(t)}}\mid\mc B]
		&\le
		\exp\left\{
		\frac{\lambda}{1-\lambda a}
		\left(
		\Esp[\tau_0\mid\mc B]
		+\sum_{k=1}^{K^\Delta(t)}a_k
		\right)
		\right\}\\
		&\le
		\exp\left\{
		\frac{\lambda(t+a)}{1-\lambda a}
		\right\},
	\end{align}
	where the last inequality follows from
	Lemma~\ref{lem_initial_hit_bound} and
	$S_{K^\Delta(t)}\le t$.
	
	Taking expectations and applying Markov's
	inequality yields
	\[
	\Prob_\bx\{\tau_{K^\Delta(t)}>M\}
	\le
	\exp\left\{
	\lambda\left(
	\frac{t+a}{1-\lambda a}-M
	\right)
	\right\},
	\]
	which proves~\eqref{eq_embedding_tail}.
\end{proof}

	\section{Regularity estimates}
	\label{sec_regularity}

The main objective of this section is to show the implication
\begin{equation}
	\label{eq:regularity_implication}
	\text{Condition~\ref{cond_speed}}
	\;\Longrightarrow\;
	\text{Condition~\ref{cond_Lp}}.
\end{equation}
Consequently, the main convergence theorem
(Theorem~\ref{thm_main}) also holds under
Condition~\ref{cond_speed}.

The proof proceeds in two steps:
\begin{itemize}
	\item First, we establish the required regularity
	estimates for diffusions on star-graphs with natural
	boundaries, using time-change and excursion-flipping
	arguments.
	
	\item We then extend these estimates to general
	finite metric graphs by decomposing diffusion paths
	into successive stopped star-graph segments and
	applying the strong Markov property.
\end{itemize}

We begin by reformulating Condition~\ref{cond_speed}
in an equivalent form that is more convenient
for our analysis.

	\begin{condition}
		\label{cond_alt0}
		There exist $k'_1>0$ and $k_2\ge0$ such that
		\begin{equation}
			\label{eq_condition_alt0}
			\widehat m^{\bv}_e(\rd y)\ge m_0(\rd y)
			\coloneqq\frac{k'_1}{1+k_2y^2}\,\rd y,
			\qquad \bv\in V,\quad e\in E(\bv),\quad 0<y<l_e,
		\end{equation}
		where the reoriented speeds are defined in
		Section~\ref{subsec_general}. The inequality is an inequality of measures.
	\end{condition}
	
	\begin{lemma}[Equivalence of the speed lower bounds]
		\label{lem_speed_conditions}
		On a graph with finitely many original edges,
		Conditions~\ref{cond_speed} and~\ref{cond_alt0} are equivalent,
		allowing the constants to change. More precisely, put
		$L_*:=\max\{l_e:e\in E,\ l_e<\infty\}$, with $L_*=0$ if the
		set is empty. If Condition~\ref{cond_speed} holds with
		$(k_1,k_2)$, then Condition~\ref{cond_alt0} holds with
		\begin{equation}\label{eq_reoriented_speed_constants}
			k'_1=\frac{k_1}{1+k_2L_*^2}\qquad \text{(with }k_2\text{ unchanged).}
		\end{equation}
		Conversely, Condition~\ref{cond_alt0} with $(k'_1,k_2)$
		implies Condition~\ref{cond_speed} with the same constants.
	\end{lemma}
	
	\begin{proof}
		If $e$ points outward from $\bv$, then
		$\widehat m_e^{\bv}=m_e$, and the original
		lower bound holds directly.
		
		If $e$ points inward to
		$\bv$, then $l_e<\infty$ and the reoriented measure is the
		pushforward of $m_e$ under $x\mapsto l_e-x$. Consequently,
		\begin{equation*}
			\widehat m_e^{\bv}(\rd y)
			\ge\frac{k_1}{1+k_2(l_e-y)^2}\,\rd y
			\ge\frac{k_1}{1+k_2L_*^2}\,\rd y
			\ge\frac{k'_1}{1+k_2y^2}\,\rd y.
		\end{equation*}
		For the converse, each original edge has its coordinate-zero
		endpoint $\bv_e=(e,0)\in V$, at which it points outward.
		Applying Condition~\ref{cond_alt0} at $\bv_e$ gives exactly
		the original-coordinate inequality of Condition~\ref{cond_speed}.
	\end{proof}

	Under this equivalent condition, the following holds.

\begin{proposition}
	\label{prop_Lp}
	Let $\mX$ be an NSE general diffusion on a finite
	metric graph $\Gamma$, defined by
	\ref{item_analytical1}--\ref{item_analytical2}.
	Assume that $\mX$ satisfies
	Condition~\ref{cond_alt0} for some
	$k'_1>0$ and $k_2\ge0$.
	
	Then, for every $p\ge1$ and $T>0$,
	the following hold:
	\begin{enuroman}
		\item \label{item_Lp}
		There exists a constant $M_{p,T}>0$
		such that, for all $0\le s\le t\le T$,
		\begin{equation}
			\left\|
			\sup_{u\in[s,t]}
			d(\mX_u,\mX_s)
			\right\|_{L^p(\Prob_{\bx})}
			\le
			M_{p,T}(1+k_2x)\sqrt{t-s}.
		\end{equation}
		
		\item \label{item_regularity}
		For every $\alpha\in[0,\frac12)$,
		there exists a constant
		$C_{p,\alpha,T}>0$ such that
		\begin{equation}
			\left\|
			\sup_{0\le s<t\le T}
			\frac{d(\mX_t,\mX_s)}
			{|t-s|^\alpha}
			\right\|_{L^p(\Prob_{\bx})}
			\le
			C_{p,\alpha,T}(1+k_2x).
		\end{equation}
	\end{enuroman}
	Both constants may depend on the fixed diffusion
	and graph, but are independent of the starting
	point $\bx=(e,x)\in\Gamma$.
\end{proposition}

	The implication~\ref{item_Lp} $\Rightarrow$~\ref{item_regularity}
	follows from the metric-valued Kolmogorov continuity theorem, applied
	in the \emph{metric completion} of $\Gamma$. This qualification is
	needed because $\Gamma$ may contain finite open natural endpoints,
	so its geodesic metric need not be complete. We recall the theorem first.
	
	\begin{theorem}
		\label{thm_KC}
		[Kolmogorov, see \cite{friz2010multidimensional}, Theorem~A.10]
		Let $X$ be a stochastic process taking values in a Polish space $(\mc S,d) $,
		defined on a probability space $(\Omega, \process{\bF_t},\Prob) $.
		Assume that for some $M>0 $ and $ p>r\ge 1$, it holds that
		\begin{equation}
			\xnorm{d(X_s,X_t)}{L^{p}(\Prob)} \le M \abs{t-s}^{1/r}, \quad \text{for all }
			s,t \in [0,T].
		\end{equation}
		Then, for any $\alpha \in [0, \frac 1r - \frac 1p) $, there exists an explicit constant
		$C=C(r,p,\alpha,T)>0 $, depending only on the displayed parameters, such that
		\begin{equation}
			\xnorm{\sup_{\substack{s\not = t\\ s,t\le T}} \frac{d(X_t, X_s)}{|t-s|^{\alpha}}}{L^p(\Prob)} \le CM.
		\end{equation}
	\end{theorem}

	\begin{remark}[Applying Kolmogorov's theorem on the completed graph]
		\label{rmk_metric_completion}
		Write $(\overline\Gamma,\overline d)$ for the metric completion
		of $(\Gamma,d)$. Since $\Gamma$ has finitely many edges and
		is separable, $\overline\Gamma$ is complete and separable,
		hence Polish. The canonical inclusion
		$\iota:\Gamma\hookrightarrow\overline\Gamma$ is an isometry:
		$\overline d(\iota(z),\iota(w))=d(z,w)$ for $z,w\in\Gamma$.
		We apply Theorem~\ref{thm_KC} to $\iota(\mX)$, with exactly
		the same distance-increment bounds. The resulting H\"older
		seminorm estimate is therefore the asserted estimate for $\mX$
		with the original geodesic distance. No completion point is added
		to the state space of the diffusion: its paths remain
		$\Gamma$-valued in the conservative setting.
	\end{remark}

	\begin{remark}
		Although not explicitly stated in Proposition \ref{prop_Lp}, the constant $M$  depends intrinsically on the law of the diffusion process $\mX $. More precisely, our proof yields explicit expressions for this constant in terms of the values of $(k_1',k_2)$ from Condition~\ref{cond_alt0} and the bias parameters $(\beta^{\bv}_e;\,e\in E,\, \bv\in V) $.
	\end{remark}

	\subsection{A sufficient condition: the star-graph case} 
	\label{subsec_sufficient*}
	
	In this section, we consider $\mX $ a general diffusion on the star-graph $\Gamma $ as defined in the prelude of Section~\ref{sec_transition} with additional assumptions on its scales and speeds.
	First, assume $\mX $ is NSE, that is: $s_e(x)=x $, for all $x\in(0,l_e) $ and $e\in E $.
	We further assume the speed measures satisfy Condition~\ref{cond_speed} for constants $k_1$ and $k_2 $, and that the boundaries of $\Gamma $ are all natural for $\mX $.
	By the one-dimensional Feller tests and local equivalence in law,
	the outer endpoint $l_e$ of each ray is natural precisely when both
	(right-end) integrals in~\eqref{eq: boundary_classification} diverge.
	Since $s_e(y)=y$, fixing any $c_e\in(0,l_e)$ gives the explicit
	conditions
	\begin{equation}\label{eq_star_natural_Feller}
		\begin{aligned}
		I_{1,e}&:=\lim_{b\uparrow l_e}
		\int_{c_e}^{b}m_e((c_e,y))\,\rd y=\infty,\\
		I_{2,e}&:=\lim_{b\uparrow l_e}
		\int_{(c_e,b)}(y-c_e)\,m_e(\rd y)=\infty,
		\qquad e\in E.
		\end{aligned}
	\end{equation}
	
	Under these conditions, the following result holds. 
	
	\begin{proposition}
	[Star-graph version of Proposition~\ref{prop_Lp}]
	\label{prop_Lp_star}
	Assume $\mX $ is defined 
	on the probability space $\mc P_{\bx}\coloneqq (\Omega, \process{\bF_t},\Prob_{\bx}) $ such that
	$\mX_{0}=\bx $, $\Prob_{\bx} $-almost surely.
	\begin{enuroman}
		\item \label{item_lp_star} For all $p\ge 1 $ and $T>0 $, there exists an explicit constant $M = M (p,T)  >0 $ such that
		\begin{equation}
			\label{eq_Lp_star}
			\xnorm{\sup_{u\in [s,t]}d(\mX_u, \mX_s)}{L^{p}(\Prob_{\bx})} \le 
			M (1+ k_2 x) \sqrt{t-s},
			\quad \text{for all } s,t \in [0,T].
		\end{equation}
		
		\item \label{item_KC_star}	For any $\alpha \in [0, \frac 12 - \frac 1p) $, there exists an explicit constant
		$C=C(p,\alpha,T) >0 $ non-increasing in $p$ such that
		\begin{equation}
			\xnorm{ \sup_{\substack{s\not = t\\ s,t\le T}} \frac{d( \mX_t, \mX_s)}{|t-s|^{\alpha}} }{L^{p}(\Prob_{\bx})} \le C M (1 + k_2x).
		\end{equation}
	\end{enuroman} 
\end{proposition}

The main argument relies on the time-change
characterization of~\cite{anagnostakis2025walsh},
which allows us to reduce the problem to Walsh
Brownian motion and, subsequently, to the
one-dimensional setting through the excursion-flipping
construction relating Walsh and skew Brownian motions.
We then employ techniques similar to those
in~\cite[Section~5.1]{anagnostakis2023general},
while correcting some technical points in the
arguments of that reference.

\begin{proof}
	We distinguish the cases $\#E=1$ and $\#E\ge2$.
	
	\medskip
	
	\noindent
	\emph{Case $\#E=1$.}
	In this case, $\Gamma$ is a half-line, possibly
	with a finite natural endpoint. By the
	time-change characterization of one-dimensional
	diffusions and the excursion-flipping construction,
	we have the identity in law
	\[
	\mX\eqlaw |Z|,
	\]
	where $Z$ is a diffusion on $(-l_e,l_e)$,
	on natural scale, with speed measure
	\begin{equation}
		m_Z(\rd y)=
		\begin{cases}
			m_e(\rd y), & y>0,\\
			m_e(-\rd y), & y<0,\\
			\rho\,\delta_0(\rd y), & y=0.
		\end{cases}
	\end{equation}
	Here the measure on the negative half-line
	is the pushforward of $m_e$ under reflection.
	The representation follows from
	\cite[Theorem~V.47.1 and the following
	Remark~(ii)]{RogWilV2} and the
	excursion-flipping arguments in
	\cite[pp.~115--116]{ItoMcKean96} and
	\cite[Example~5.7]{salisbury1986construction}.
	
	Since $m_e$ satisfies Condition~\ref{cond_speed},
	its reflected extension satisfies
	\[
	m_Z(\rd y)
	\ge \frac{k_1}{1+k_2y^2}\,\rd y,
	\qquad y\in(-l_e,l_e).
	\]
	The nonnegative atom at zero does not affect
	this lower bound.
	
	Let $Z_0=x\ge0$. By
	\cite[Theorem~3.1]{ankirchner2021wasserstein},
	for every $p\ge1$ and $T>0$ there exists
	$c_{p,T}>0$ such that
	\begin{align}
		&\left\|
		\sup_{u\in[s,t]}d(\mX_u,\mX_s)
		\right\|_{L^p(\Prob_{\bx})}
		\notag\\
		&\quad=
		\left\|
		\sup_{u\in[s,t]}
		\big||Z_u|-|Z_s|\big|
		\right\|_{L^p(\Qrob_x)}
		\notag\\
		&\quad\le
		\left\|
		\sup_{u\in[s,t]}|Z_u-Z_s|
		\right\|_{L^p(\Qrob_x)}
		\notag\\
		&\quad\le
		c_{p,T}(1+k_2x)\sqrt{t-s},
		\qquad 0\le s<t\le T.
	\end{align}
	This proves Assertion~\ref{item_lp_star}
	when $\#E=1$.
	
	\medskip
	
	\noindent
	\emph{Case $\#E\ge2$.}
	We first explain how the argument applies
	when some rays have finite natural endpoints.
	
	Let $\Gamma_\infty$ be the star graph obtained
	by extending every finite ray of $\Gamma$ to
	an infinite ray, and let $d_\infty$ be its
	geodesic metric. Consider a Walsh Brownian
	motion $\mW=(J,R)$ on $\Gamma_\infty$ with
	spinning probabilities $(\beta_e)_{e\in E}$.
	
	Let $\sigma$ be the first time that $\mW$
	reaches one of the finite outer endpoints
	of the original graph, with $\sigma=\infty$
	if no such endpoint exists.
	
	By the time-change characterization in
	\cite[Theorem~2.5]{anagnostakis2025walsh},
	localized to the original rays, we may
	realize $\mX$ as
	\[
	\mX_t=\mW_{\gamma(t)},
	\qquad t\ge0,
	\]
	where $\gamma$ is the right-inverse of
	\begin{equation}
		\label{eq_thm_A_timechange}
		A(u)=
		\sum_{e\in E}
		\int_{(0,l_e)}
		\int_0^u
		\indic{J(v)=e}
		\,\rd\loct{R}{y}{v}\,
		m_e(\rd y)
		+\frac{\rho}{2}\loct{R}{0}{u},
		\qquad u<\sigma.
	\end{equation}
	The local times are normalized consistently
	with the occupation-density formula and
	the time-change representation.
	
	If a finite outer endpoint is natural,
	the corresponding Feller tests imply that
	the time-changed process cannot reach it
	in finite time. Equivalently, the clock
	diverges as the underlying Walsh motion
	approaches such an endpoint. We may
	therefore extend $A$ by setting
	$A(u)=\infty$ for $u\ge\sigma$.
	Its inverse remains strictly below
	$\sigma$ on finite time horizons.
	
	This construction uses the full-ray Walsh
	motion only as an auxiliary process;
	no extension of the original speed measures
	through their singular natural endpoints
	is required.
	
	Define
	\begin{equation}
		\label{eq_thm_A_timechange0}
		A_0(u)
		:=\int_{(0,\infty)}
		\loct{R}{y}{u}
		\frac{k_1}{1+k_2y^2}\,\rd y
		=\int_0^u
		\frac{k_1}{1+k_2R_v^2}\,\rd v,
		\qquad u\ge0,
	\end{equation}
	and let $\gamma_0$ be its right-inverse.
	
	Condition~\ref{cond_speed} and $\rho\ge0$
	give $A(u)\ge A_0(u)$ for $u<\sigma$.
	With the preceding extension of $A$,
	this inequality holds for all $u\ge0$.
	Consequently,
	\[
	\gamma(t)\le\gamma_0(t),
	\qquad t\ge0.
	\]
	Since the time-changed process remains
	on the original graph, we obtain
	\begin{equation}
		\label{eq_majoration_tc_property}
		\begin{aligned}
			\sup_{0\le u\le t}
			d(\mX_u,\bx)
			&=
			\sup_{0\le v\le\gamma(t)}
			d_\infty(\mW_v,\bx)\\
			&\le
			\sup_{0\le v\le\gamma_0(t)}
			d_\infty(\mW_v,\bx).
		\end{aligned}
	\end{equation}
	
	Fix an initial point $\bx=(e,x)$.
	Since $\#E\ge2$ and all spinning
	probabilities are strictly positive,
	we have $\beta_e\in(0,1)$.
	
	By the excursion-flipping construction
	of skew and Walsh Brownian motions
	(see \cite[pp.~115--116]{ItoMcKean96},
	\cite[Example~5.7]{salisbury1986construction},
	the afterword of~\cite{walsh1978diffusion},
	and the introduction of~\cite{barlow1989onwalsh}),
	the process
	\begin{equation}
		B_t^{\beta_e}:=
		\begin{cases}
			R_t, & J(t)=e,\\
			-R_t, & J(t)\ne e,
		\end{cases}
		\qquad t\ge0,
	\end{equation}
	is a skew Brownian motion with bias
	parameter $\beta_e$ and initial value
	$B_0^{\beta_e}=x$.
	
	Moreover, by the geometry of the star graph,
	\begin{equation}
		d_\infty(\mW_t,\bx)
		=|B_t^{\beta_e}-x|,
		\qquad t\ge0.
	\end{equation}
	Since $R_t=|B_t^{\beta_e}|$,
	the occupation-time representation
	in~\eqref{eq_thm_A_timechange0} becomes
	\begin{equation}
		\label{eq_thm_A_timechange0_skew}
		A_0(u)
		=\int_0^u
		\frac{k_1}
		{1+k_2(B_v^{\beta_e})^2}\,\rd v.
	\end{equation}
	
	With our scaling conventions, the scale
	function and speed measure of
	$B^{\beta_e}$ are
	(see \cite[Appendix~1.12]{BorSal})
	\[
	s_{\beta_e}(y)=\frac{y}{a_e(y)},
	\qquad
	m_{\beta_e}(\rd y)=a_e(y)\,\rd y,
	\]
	where
	\begin{equation}
		a_e(y)=
		\begin{cases}
			2\beta_e, & y>0,\\
			2(1-\beta_e), & y\le0.
		\end{cases}
	\end{equation}
	
	The process
	\[
	Y_t:=s_{\beta_e}(B_t^{\beta_e})
	\]
	is a diffusion on $\mathbb R$ on natural
	scale, with speed measure
	\begin{equation}
		m_Y(\rd y)=a_e(y)^2\,\rd y.
	\end{equation}
	
	By \cite[Theorem~V.47.1]{RogWilV2},
	after enlarging the probability space
	if necessary, we may write
	\[
	Y_t=W_{\gamma_1(t)},
	\]
	where $W$ is a standard Brownian motion
	started at $s_{\beta_e}(x)$, and
	$\gamma_1$ is the right-inverse of
	\begin{equation}
		A_1(u):=
		\int_{\mathbb R}
		\loct{W}{y}{u}\,
		a_e(y)^2\,\rd y
		=\int_0^u a_e(W_v)^2\,\rd v.
	\end{equation}
	
	Let $q_{\beta_e}=s_{\beta_e}^{-1}$.
	Then
	\[
	q_{\beta_e}(y)=a_e(y)y.
	\]
	This function is globally Lipschitz.
	More precisely,
	\[
	q_*|y-z|
	\le
	|q_{\beta_e}(y)-q_{\beta_e}(z)|
	\le q^*|y-z|,
	\qquad y,z\in\mathbb R,
	\]
	where
	\begin{equation}
		\begin{aligned}
			q_*&:=
			2\min_{i\in E}
			\{\beta_i\wedge(1-\beta_i)\}>0,\\
			q^*&:=
			2\max_{i\in E}
			\{\beta_i\vee(1-\beta_i)\}\le2.
		\end{aligned}
	\end{equation}
	
	Since
	$B_t^{\beta_e}=q_{\beta_e}(Y_t)$,
	we have
	\begin{align}
		\label{eq_majoration_tc_property2}
		d_\infty(\mW_{\gamma_0(t)},\bx)
		&=
		|B_{\gamma_0(t)}^{\beta_e}-x|
		\notag\\
		&=
		\left|
		q_{\beta_e}
		\bigl(W_{\gamma_1(\gamma_0(t))}\bigr)
		-q_{\beta_e}(s_{\beta_e}(x))
		\right|
		\notag\\
		&\le
		q^*
		\left|
		W_{\gamma_1(\gamma_0(t))}
		-s_{\beta_e}(x)
		\right|.
	\end{align}
	
	We now identify the composed time change.
	Since $A_1$ is continuous and strictly
	increasing, the process
	\[
	Z_t:=W_{\gamma_1(\gamma_0(t))}
	\]
	is obtained from $W$ by the inverse
	of $A_0\circ A_1$.
	
	Indeed, using
	$B_{A_1(v)}^{\beta_e}
	=q_{\beta_e}(W_v)$
	and the change of variables
	$u=A_1(v)$, we obtain
	\begin{align}
		A_0(A_1(t))
		&=
		\int_0^{A_1(t)}
		\frac{k_1}
		{1+k_2(B_u^{\beta_e})^2}
		\,\rd u
		\notag\\
		&=
		\int_0^t
		\frac{k_1}
		{1+k_2q_{\beta_e}(W_v)^2}
		a_e(W_v)^2\,\rd v
		\notag\\
		&=
		\int_0^t
		\frac{k_1a_e(W_v)^2}
		{1+k_2a_e(W_v)^2W_v^2}
		\,\rd v.
	\end{align}
	
	Thus $Z$ is a diffusion on
	$\mathbb R$, on natural scale, with
	speed measure
	\begin{equation}
		\begin{aligned}
			m_Z(\rd y)
			&=
			\frac{k_1a_e(y)^2}
			{1+k_2a_e(y)^2y^2}\,\rd y\\
			&\ge
			\frac{k_1q_*^2}
			{1+k_2(q^*)^2y^2}\,\rd y.
		\end{aligned}
	\end{equation}
	Its initial value is
	$Z_0=s_{\beta_e}(x)$.
	
	Therefore,
	\cite[Theorem~3.1]{ankirchner2021wasserstein}
	applies to $Z$.
	For every $p\ge1$ and $T>0$,
	there exists $M_0=M_0(p,T)>0$
	such that
	\begin{equation}
		\begin{aligned}
			&\left\|
			\sup_{u\in[s,t]}|Z_u-Z_s|
			\right\|_{L^p(\Prob_{\bx})}\\
			&\qquad\le
			M_0\left(
			1+k_2(q^*)^2s_{\beta_e}(x)
			\right)\sqrt{t-s},
		\end{aligned}
	\end{equation}
	for $0\le s<t\le T$.
	The constant $M_0$ depends on the
	speed lower-bound parameters, but
	not on the initial point.
	
	Since $x\ge0$,
	\[
	s_{\beta_e}(x)
	=\frac{x}{2\beta_e}
	\le\frac{x}{q_*}.
	\]
	Combining
	\eqref{eq_majoration_tc_property}
	and~\eqref{eq_majoration_tc_property2},
	we obtain
	\begin{align}
		\label{eq_inequation_distance}
		&\left\|
		\sup_{0\le u\le t}
		d(\mX_u,\bx)
		\right\|_{L^p(\Prob_{\bx})}
		\notag\\
		&\quad\le
		q^*M_0
		\left(
		1+\frac{k_2(q^*)^2}{q_*}x
		\right)\sqrt t
		\notag\\
		&\quad\le
		C_0(1+k_2x)\sqrt t,
	\end{align}
	where we may take
	\[
	C_0=q^*M_0
	\max\left\{
	1,\frac{(q^*)^2}{q_*}
	\right\}.
	\]
	Because the star graph has finitely
	many rays, this constant is uniform
	over the choice of the initial edge.
	
	The estimate also holds when the
	process starts at the central vertex
	$\bv$, with $x=0$.
	
	It remains to establish the estimate
	for arbitrary increments
	$0\le s<t\le T$.
	
	Write
	\[
	R_s=d(\bv,\mX_s).
	\]
	By the Markov property at time $s$
	and~\eqref{eq_inequation_distance},
	\begin{align*}
		&\left\|
		\sup_{u\in[s,t]}
		d(\mX_u,\mX_s)
		\right\|_{L^p(\Prob_{\bx})}\\
		&\qquad\le
		C_0\sqrt{t-s}\,
		\left\|1+k_2R_s
		\right\|_{L^p(\Prob_{\bx})}.
	\end{align*}
	
	Since
	\[
	R_s\le x+d(\bx,\mX_s),
	\]
	Minkowski's inequality and
	\eqref{eq_inequation_distance} give
	\begin{align*}
		\left\|1+k_2R_s\right\|_{L^p}
		&\le
		1+k_2x+
		k_2\left\|d(\bx,\mX_s)\right\|_{L^p}\\
		&\le
		1+k_2x+
		k_2C_0(1+k_2x)\sqrt s\\
		&\le
		(1+k_2x)(1+k_2C_0\sqrt T).
	\end{align*}
	
	Consequently,
	\begin{align}
		&\left\|
		\sup_{u\in[s,t]}
		d(\mX_u,\mX_s)
		\right\|_{L^p(\Prob_{\bx})}
		\notag\\
		&\quad\le
		M(1+k_2x)\sqrt{t-s}
		\notag\\
		&\quad\le
		M(1+8k_2x)\sqrt{t-s},
	\end{align}
	where
	\[
	M=C_0(1+k_2C_0\sqrt T).
	\]
	This proves Assertion~\ref{item_lp_star}
	for $\#E\ge2$ and completes the
	proof of the first assertion.
	
	Finally, Assertion~\ref{item_KC_star}
	follows from Assertion~\ref{item_lp_star}
	and the metric-valued Kolmogorov
	continuity theorem
	(Theorem~\ref{thm_KC}),
	applied in the metric completion
	of $\Gamma$.
\end{proof}
	
	\subsection{Proof of Proposition~\ref{prop_Lp} (the general case)} 
	\label{subsec_sufficient}

We first establish a uniform lower bound on the
probability that the diffusion remains near an original
vertex for a fixed time. This yields the geometric
tail bound for the number of transitions between
distinct original vertices, which is used in the
subsequent stopped-segment decomposition.

\begin{lemma}[Uniform residence time in a vertex neighborhood]
\label{lem_uniform_residence}
Let $\mX$ be an NSE diffusion satisfying
Condition~\ref{cond_alt0}. Fix $r_*>0$ small enough that, for each
$\bv\in V$, $B(\bv,r_*)$ is a star neighborhood containing no other
original vertex and no edge endpoint distinct from $\bv$. Such a common radius exists
because the original edge set is finite and all its lengths are positive.
Put
\begin{equation}\label{eq_uniform_residence_constants}
\kappa_*:=\frac{k'_1}{1+k_2r_*^2},\qquad
q_T:=\mathbb P\left\{\sup_{0\le u\le T/\kappa_*}|B_u|<r_*\right\}>0,
\end{equation}
where $B$ is a standard Brownian motion started at zero. Then, for
all $T>0$,
\begin{equation}\label{eq_uniform_residence_bound}
\inf_{\bv\in V}\Prob_{\bv}
\left\{\inf\{t\ge0:d(\mX_t,\bv)\ge r_*\}>T\right\}\ge q_T.
\end{equation}
The same estimate holds for any auxiliary NSE star diffusion with
identical local speed measures and vertex transmission parameters
inside $B(\bv,r_*)$.
\end{lemma}
\begin{proof}
Fix $\bv$ and stop at the first exit from $B(\bv,r_*)$.
By localization, the stopped diffusion agrees in law with a
star-graph diffusion with outward speed measures
$\widehat m_e^{\bv}$ and stickiness parameter $\rho_{\bv}$.
For example, to invoke the global star-graph time-change
representation, extend the local arms beyond $r_*$ to infinite natural
arms; this extension does not change the stopped law.
By the Walsh Brownian time-change representation used in the proof of
Proposition~\ref{prop_Lp_star}, the radial coordinate of the
underlying Walsh motion is $R_u=|B_u|$ and the time-change functional
$A^{\bv}$ is nonnegative and increasing. Condition~\ref{cond_alt0},
the occupation-density formula and the nonnegative sticky-vertex
contribution imply that, for
$u<H_{r_*}:=\inf\{w\ge0:R_w\ge r_*\}$,
\begin{equation}\label{eq_residence_timechange_lower}
A^{\bv}(u)\ge
\int_0^u\frac{k'_1}{1+k_2R_w^2}\,\rd w
\ge\kappa_*u.
\end{equation}
On the event $\{\sup_{u\le T/\kappa_*}R_u<r_*\}$,
we have $H_{r_*}>T/\kappa_*$ and
$A^{\bv}(T/\kappa_*)\ge T$. Hence the inverse clock
$\gamma^{\bv}(T)\le T/\kappa_*$, and the time-changed motion
cannot exit the radius-$r_*$ neighborhood before $T$.
The probability of this Brownian small-ball event is $q_T>0$,
independent of $\bv$, its spinning probabilities and its stickiness.
Only the local data in this neighborhood enter the argument,
which proves the last assertion as well.
\end{proof}

We next use this residence-time estimate to control the number of
successive visits to distinct original vertices.
	
	\begin{lemma}
		\label{lem_proba_bound}
		Let $\tau_0^V:=\inf\{t\ge0:\mX_t\in V\}$
		be the first visit to the original vertex set and,
		on $\{\tau_k^V<\infty\}$, define recursively
		\[
		\tau_{k+1}^V
		:=\inf\{t>\tau_k^V:
		\mX_t\in V\setminus\{\mX_{\tau_k^V}\}\}.
		\]
		These are successive visits to distinct original
		vertices, defined without reference to any
		subdivision of $\Gamma$. Put
		\begin{equation}
			M(s,t)\coloneqq
			\#\{k\ge0: \tau_k^V\in[s,t]\},
			\qquad 0\le s\le t.
		\end{equation}
		For all $T>0 $, $\alpha>0 $, and $\gamma\in (0,1) $, there exists some $C_{\alpha,\gamma,T} >0 $ such that
		\begin{equation}
			\sum_{N=1}^{\infty} N^{\alpha} \braces{\Prob_{\bx} \braces{M(s,t)=N }}^{\gamma} \le C_{\alpha,\gamma,T}, \quad \text{for all } \bx \in \Gamma \text{ and } 0<s<t<T.
		\end{equation}
	\end{lemma}

	\begin{proof}
		If no edge joins distinct original vertices, the successive
		distinct-vertex skeleton makes no transition and $M(s,t)\le1$,
		so the estimate is immediate. Otherwise set $p_T:=1-q_T<1$,
		with $q_T$ from
		Lemma~\ref{lem_uniform_residence}. A transition from an original
		vertex $\bv$ to a different original vertex must first leave
		$B(\bv,r_*)$; therefore, by~\eqref{eq_uniform_residence_bound},
		its probability of completion within any time interval of length
		at most $T$ is bounded above by $p_T$, uniformly in $\bv$.
		After the first vertex hit in $[s,t]$,
		all further hits of distinct vertices require a complete
		passage starting at a vertex. Repeated application of the
		strong Markov property at these successive hitting times yields
		\begin{equation}
		\Prob_\bx\{M(s,t)\ge N\}\le p_T^{N-1},\qquad N\ge1,
		\quad 0\le s<t\le T.
		\end{equation}
		Therefore $\Prob_\bx\{M(s,t)=N\}\le p_T^{N-1}$,
		and for every $\alpha>0$, $\gamma\in(0,1)$,
		\begin{equation}
			\sum_{N=1}^\infty
			N^\alpha\Prob_\bx\{M(s,t)=N\}^\gamma
			\le
			\sum_{N=1}^\infty
			N^\alpha p_T^{\gamma(N-1)}
			<\infty,
		\end{equation}
		since $p_T\in(0,1)$.
	\end{proof}

We prove Proposition~\ref{prop_Lp} by decomposing
the paths of $\mX$ into stopped star-graph
segments based on the original vertex set $V$.
This decomposition does not use any subdivision
thinness. We combine
Proposition~\ref{prop_Lp_star} with the
uniform residence-time bound of
Lemma~\ref{lem_uniform_residence}.

\begin{proof}[Proof of Proposition~\ref{prop_Lp}]
	We decompose the paths of $\mX$ into successive
	segments governed by stopped star-graph diffusions.
	The identification of these segments follows from
	localization and the strong Markov property of
	$\mX$.
	
	For each $\bv\in V$, let $\Gamma_{\bv}$
	be the star region consisting of $\bv$
	and the open interiors of all edges
	incident to $\bv$, up to, but excluding,
	their other original vertices.
	Since self-loops are excluded, every
	finite arm joining two original vertices
	has its other endpoint at a distinct
	vertex. Infinite arms and finite open
	natural ends do not give switching
	locations.
	
	Let $\widehat\Gamma_{\bv}$ be the star
	graph obtained by extending the arms
	ending at another original vertex
	beyond that vertex to infinity, while
	retaining any inaccessible natural end.
	Choose an NSE diffusion $Z^{\bv}$ on
	$\widehat\Gamma_{\bv}$ whose speed measures
	agree with those of $\mX$ on
	$\Gamma_{\bv}$ and admit extensions
	satisfying~\eqref{eq_condition_alt0}
	along the newly extended arms.
	The transmission parameters at the
	central vertex remain those of $\mX$.
	
	Write $B_{\bv}$ for the finite set
	of points corresponding to adjacent
	original vertices, and let
	\[
	\zeta_{\bv}
	=\inf\{t\ge0:Z^{\bv}_t\in B_{\bv}\}.
	\]
	By localization and uniqueness of the
	stopped diffusion determined by the
	edgewise scale and speed measures and
	the transmission condition at $\bv$,
	for every $z\in\Gamma_{\bv}$ the
	process $\mX$ started at $z$ and stopped
	upon first reaching $V\setminus\{\bv\}$
	has the same law as
	$Z^{\bv}$ started at $z$ and stopped
	at $\zeta_{\bv}$.
	This identification includes the
	exit point.
	
	Choose a measurable map
	$\mathfrak v:\Gamma\to V$ such that
	$z\in\Gamma_{\mathfrak v(z)}$
	for every $z\in\Gamma$, with
	$\mathfrak v(\bv)=\bv$ at original
	vertices. On an edge interior joining
	two original vertices, either endpoint
	may be selected.
	
	For the diffusion started at $\bx$,
	set
	\[
	\sigma_0=0,
	\qquad
	\bv_0=\mathfrak v(\bx),
	\]
	and define recursively
	\begin{equation}
		\label{eq_piecing_switching_times}
		\begin{aligned}
			\sigma_{j+1}
			&=\inf\{t>\sigma_j:
			\mX_t\in V\setminus\{\bv_j\}\},\\
			\bv_{j+1}
			&=\mX_{\sigma_{j+1}}
			\quad\text{on }\{\sigma_{j+1}<\infty\}.
		\end{aligned}
	\end{equation}
	The recursion stops if a switching
	time is infinite.
	
	By the strong Markov property at
	$\sigma_j$ and the preceding stopped-law
	identification, conditional on the
	starting point of each segment, its
	law up to $\sigma_{j+1}$ is that of
	the corresponding stopped auxiliary
	star diffusion $Z^{\bv_j}$.
	In particular, the same identification
	applies to the remaining portion of
	a segment after any deterministic
	time $s$, with the active star
	determined by the path history up to $s$.
	
	Let $N(s,t)$ denote the number of
	switching times in $(s,t]$.
	We first control this number.
	
	After every finite switch, the process
	starts at an original vertex.
	A further switch cannot occur before
	the process exits the radius-$r_*$
	neighborhood of that vertex.
	By Lemma~\ref{lem_uniform_residence},
	for every $T>0$ the probability of
	such an exit within time $T$ is at most
	\[
	p_T:=1-q_T<1,
	\]
	uniformly over original vertices.
	
	Conditioning successively at the
	switching times therefore gives
	\begin{equation}
		\label{eq_piecing_tail}
		\Prob_\bx\{N(s,t)\ge n\}
		\le p_T^{n-1},
		\qquad
		n\ge1,\quad 0\le s<t\le T.
	\end{equation}
	Here the first switch in $(s,t]$
	is unrestricted, while each of the
	remaining $n-1$ switches requires
	a complete passage starting at an
	original vertex.
	
	In particular, the switching times
	cannot accumulate in finite time.
	If no edge joins distinct original
	vertices, no finite switch is possible,
	and the estimates below follow directly
	from the star-graph case.
	
	Fix $r\ge1$.
	For each auxiliary star graph
	$\widehat\Gamma_{\bv}$, let
	$d_{\bv}$ be its geodesic distance
	and set
	\[
	r_{\bv}(z)=d_{\bv}(\bv,z),
	\qquad z\in\Gamma_{\bv}.
	\]
	Proposition~\ref{prop_Lp_star} gives
	\begin{equation}
		\label{eq_star_piece_estimate}
		\left\|
		\sup_{0\le u\le h}
		d_{\bv}(Z^{\bv}_u,z)
		\right\|_{L^r(\widehat\Prob_z^{\bv})}
		\le
		B_{r,T}(1+k_2r_{\bv}(z))\sqrt h,
		\qquad 0<h\le T.
	\end{equation}
	Since $V$ is finite, the constant
	$B_{r,T}$ may be chosen uniformly
	over $\bv\in V$.
	
	Let $o\in V$ be fixed, and let
	$L_*$ be the maximum finite
	original-edge length, taking
	$L_*=0$ if there are none.
	For every $z\in\Gamma_{\bv}$,
	\[
	r_{\bv}(z)\le d(o,z)+L_*.
	\]
	Indeed, on a finite arm the outward
	coordinate is bounded by $L_*$,
	whereas on an infinite ray it is
	bounded by $d(o,z)$.
	
	Consequently, after enlarging
	$B_{r,T}$ if necessary,
	\begin{equation}
		\label{eq_piecing_star_bound}
		\left\|
		\sup_{0\le u\le h}
		d_{\bv}(Z^{\bv}_u,z)
		\right\|_{L^r(\widehat\Prob_z^{\bv})}
		\le
		B_{r,T}(1+d(o,z))\sqrt h.
	\end{equation}
	By the stopped-law identification,
	this estimate also bounds the
	corresponding stopped segment of
	$\mX$, since the original graph
	distance does not exceed the
	auxiliary-star distance within
	that segment.
	
	We first establish a global spatial
	moment estimate on $[0,T]$.
	
	Let $Z_j(T)$ denote the maximal
	displacement of the segment of
	$\mX$ starting at $\sigma_j$,
	measured from its starting point
	over at most $T$ units of time
	and stopped at $\sigma_{j+1}$.
	Set $Z_j(T)=0$ if $\sigma_j=\infty$.
	
	For conjugate exponents $q,q'>1$,
	the stopped-law identification,
	\eqref{eq_piecing_star_bound},
	and the strong Markov property
	give
	\begin{align*}
		\|Z_0(T)\|_{L^{rq}}
		&\le
		B_{rq,T}(1+d(o,\bx))\sqrt T,\\
		\|Z_j(T)\|_{L^{rq}}
		&\le B_{rq,T}\sqrt T,
		\qquad j\ge1,
	\end{align*}
	because every segment after the
	initial one starts at an original
	vertex.
	
	Pathwise,
	\[
	\sup_{u\le T}d(\mX_u,\bx)
	\le
	\sum_{j=0}^{N(0,T)}Z_j(T).
	\]
	By convexity, H\"older's inequality,
	and~\eqref{eq_piecing_tail},
	\begin{align}
		&\Esp_\bx\left[
		\sup_{u\le T}d(\mX_u,\bx)^r
		\right]
		\notag\\
		&\quad\le
		\sum_{n\ge0}(n+1)^{r-1}
		\sum_{j=0}^n
		\Esp_\bx\left[
		\indic{N(0,T)=n}Z_j(T)^r
		\right]
		\notag\\
		&\quad\le
		C_{r,T}(1+d(o,\bx))^r
		\sum_{n\ge0}(n+1)^r
		\Prob_\bx\{N(0,T)=n\}^{1/q'}
		\notag\\
		&\quad\le
		C_{r,T}(1+d(o,\bx))^r.
	\end{align}
	The last series converges by the
	geometric bound~\eqref{eq_piecing_tail}.
	
	Applying the same estimate at
	higher moments yields
	\begin{equation}
		\label{eq_Xs_moment}
		\sup_{0\le s\le T}
		\Esp_\bx\left[
		(1+d(o,\mX_s))^{rq}
		\right]
		\le
		C_{r,q,T}(1+d(o,\bx))^{rq}.
	\end{equation}
	
	Now fix $0\le s<t\le T$ and put
	$h=t-s$.
	Decompose the path of $\mX$ on
	$[s,t]$ into the initial, possibly
	partial, star segment and the
	$N(s,t)$ subsequent segments.
	
	Let $Z_0(h)$ denote the maximal
	displacement of the initial segment
	from $\mX_s$, measured over at most
	$h$ units of time.
	For $j\ge1$, let $Z_j(h)$ denote
	the maximal displacement of the
	segment starting at the $j$-th
	switch after $s$, measured from
	its entry vertex over at most
	$h$ units of time.
	As before, the displacements are
	stopped at the ends of their
	respective segments, and a
	nonexistent segment contributes zero.
	
	The active star label at time $s$
	is $\bF_s$-measurable.
	By the Markov property at time $s$,
	the stopped-law identification,
	and~\eqref{eq_piecing_star_bound},
	\[
	\Esp_\bx\left[
	Z_0(h)^{rq}\mid\bF_s
	\right]
	\le
	B_{rq,T}^{rq}
	(1+d(o,\mX_s))^{rq}
	h^{rq/2}.
	\]
	Taking expectations and using
	\eqref{eq_Xs_moment} gives
	\begin{equation}
		\|Z_0(h)\|_{L^{rq}}
		\le
		C_{r,q,T}(1+d(o,\bx))\sqrt h.
	\end{equation}
	
	For $j\ge1$, the strong Markov
	property at the corresponding
	switching time gives
	\[
	\|Z_j(h)\|_{L^{rq}}
	\le B_{rq,T}\sqrt h,
	\]
	since every such segment starts
	at an original vertex.
	
	The triangle inequality gives
	\[
	\sup_{s\le u\le t}
	d(\mX_u,\mX_s)
	\le
	\sum_{j=0}^{N(s,t)}Z_j(h).
	\]
	Applying convexity and H\"older's
	inequality as above, together with
	\eqref{eq_piecing_tail}, yields
	\begin{align}
		&\Esp_\bx\left[
		\sup_{s\le u\le t}
		d(\mX_u,\mX_s)^r
		\right]
		\notag\\
		&\quad\le
		\sum_{n\ge0}(n+1)^{r-1}
		\sum_{j=0}^n
		\Esp_\bx\left[
		\indic{N(s,t)=n}Z_j(h)^r
		\right]
		\notag\\
		&\quad\le
		C_{r,T}(1+d(o,\bx))^r h^{r/2}
		\sum_{n\ge0}(n+1)^r
		p_T^{(n-1)_+/q'}
		\notag\\
		&\quad\le
		C_{r,T}(1+d(o,\bx))^r h^{r/2}.
	\end{align}
	Consequently,
	\begin{equation}
		\label{eq_piecing_increment}
		\left\|
		\sup_{s\le u\le t}
		d(\mX_u,\mX_s)
		\right\|_{L^r(\Prob_\bx)}
		\le
		C_{r,T}(1+d(o,\bx))
		\sqrt{t-s}.
	\end{equation}
	
	We now distinguish the cases
	$k_2>0$ and $k_2=0$.
	
	\medskip
	\noindent
	\emph{Case $k_2>0$.}
	Since $\Gamma$ has finitely many
	original edges, there is a constant
	$C_\Gamma>0$ such that
	\[
	1+d(o,\bx)\le C_\Gamma(1+x),
	\qquad \bx=(e,x).
	\]
	For fixed $k_2>0$, the quantities
	$1+x$ and $1+k_2x$ are comparable.
	Thus~\eqref{eq_piecing_increment}
	implies
	\begin{equation}
		\left\|
		\sup_{s\le u\le t}
		d(\mX_u,\mX_s)
		\right\|_{L^r(\Prob_\bx)}
		\le
		M_{r,T}(1+k_2x)\sqrt{t-s}.
	\end{equation}
	
	\medskip
	\noindent
	\emph{Case $k_2=0$.}
	In this case,
	Proposition~\ref{prop_Lp_star}
	provides a bound independent of
	the starting point:
	\[
	\left\|
	\sup_{0\le u\le h}
	d_{\bv}(Z^{\bv}_u,z)
	\right\|_{L^{rq}(\widehat\Prob_z^{\bv})}
	\le B_{rq,T}\sqrt h.
	\]
	Therefore the stopped-law
	identification and the strong
	Markov property give
	\[
	\|Z_j(h)\|_{L^{rq}}
	\le B_{rq,T}\sqrt h,
	\qquad j\ge0,
	\]
	uniformly over the initial point.
	
	Using the triangle inequality,
	convexity, H\"older's inequality,
	and~\eqref{eq_piecing_tail}
	directly gives
	\begin{align}
		&\Esp_\bx\left[
		\sup_{s\le u\le t}
		d(\mX_u,\mX_s)^r
		\right]
		\notag\\
		&\quad\le
		\sum_{n\ge0}(n+1)^{r-1}
		\sum_{j=0}^n
		\Esp_\bx\left[
		\indic{N(s,t)=n}Z_j(h)^r
		\right]
		\notag\\
		&\quad\le
		C_{r,T}h^{r/2}
		\sum_{n\ge0}(n+1)^r
		p_T^{(n-1)_+/q'}
		\notag\\
		&\quad\le C_{r,T}h^{r/2}.
	\end{align}
	Consequently,
	\[
	\left\|
	\sup_{s\le u\le t}
	d(\mX_u,\mX_s)
	\right\|_{L^r(\Prob_\bx)}
	\le
	M_{r,T}\sqrt{t-s},
	\]
	uniformly over the initial point.
	
	Combining both cases proves
	Assertion~\ref{item_Lp}
	for every $r\ge1$.
	
	Finally, Assertion~\ref{item_regularity}
	follows from Assertion~\ref{item_Lp}
	and Theorem~\ref{thm_KC}, applied
	in the metric completion of $\Gamma$.
\end{proof}

	\section{Convergence}
	\label{sec_proof}

	\subsection{Proof of Theorems~\ref{thm_main} and~\ref{thm_absorbing}}

	We first record a consequence of Condition~\ref{cond_Lp} for the
	possibly unbounded graph $\Gamma$. Fix an original vertex $o\in V$
	and write $R_t=d(o,\mX_t)$.

\begin{lemma}
	\label{lem_spatial_growth}
	Under Condition~\ref{cond_Lp}, for every $r\ge1$ there are
	constants $C_r,\kappa_r>0$ such that, for all $m\ge0$,
	\begin{equation}
		\Esp_\bx\left[\sup_{0\le s\le m}(1+R_s)^r\right]
		\le C_r(1+d(o,\bx))^r e^{\kappa_r m}.
	\end{equation}
\end{lemma}

\begin{proof}
Let $L_*$ be the maximum length among the original
edges of finite length, with $L_*=0$ if no such
edges exist.
	Since the graph
	has finitely many original edges, the coordinate $y$
	of any $z=(e,y)$ satisfies
	\[
	y\le d(o,z)+L_*.
	\]
	Indeed, this is immediate on finite edges, while on
	infinite edges the distance from $o$ to $z$ is at
	least $y$.
	
	Fix $q>\max\{r,2\}$ and
	$\alpha\in(0,\frac12-\frac1q)$.
	By Condition~\ref{cond_Lp} and
	Theorem~\ref{thm_KC}, applied on $[0,1]$,
	the $\alpha$-H\"older seminorm of $\mX$
	has an $L^q$ bound proportional to $1+y$
	when $\mX_0=z=(e,y)$.
	Since
	\[
	\sup_{0\le u\le1}d(\mX_u,z)
	\le
	\sup_{0\le s<t\le1}
	\frac{d(\mX_s,\mX_t)}{|t-s|^\alpha},
	\]
	and $y\le d(o,z)+L_*$, there exists
	$b_r>0$, independent of $z$, such that
	\begin{equation}
		\left\|\sup_{0\le u\le1}d(\mX_u,z)
		\right\|_{L^r(\Prob_z)}
		\le b_r(1+d(o,z)).
	\end{equation}
	
	Set
	\[
	Z_n=1+\sup_{0\le u\le n}R_u,
	\qquad n\ge0.
	\]
	By the triangle inequality,
	\[
	Z_{n+1}\le Z_n+
	\sup_{n\le u\le n+1}d(\mX_u,\mX_n).
	\]
	The Markov property at time $n$ gives
	\[
	\begin{aligned}
		&\Esp_\bx\left[
		\sup_{n\le u\le n+1}
		d(\mX_u,\mX_n)^r
		\,\middle|\,\bF_n\right]\\
		&\qquad\le b_r^r(1+d(o,\mX_n))^r
		\le b_r^r Z_n^r.
	\end{aligned}
	\]
	Hence Minkowski's inequality yields
	\[
	\|Z_{n+1}\|_{L^r}
	\le(1+b_r)\|Z_n\|_{L^r}.
	\]
	Iterating and using
	$Z_0=1+d(o,\bx)$, we obtain
	\[
	\|Z_n\|_{L^r}
	\le(1+b_r)^n(1+d(o,\bx)).
	\]
	
	Finally, for $n=\lceil m\rceil$,
	\[
	\begin{aligned}
		\Esp_\bx\left[
		\sup_{0\le s\le m}(1+R_s)^r\right]
		&\le \|Z_n\|_{L^r}^r\\
		&\le (1+b_r)^r(1+d(o,\bx))^r
		e^{r\log(1+b_r)m}.
	\end{aligned}
	\]
	The assertion follows with
	\[
	C_r=(1+b_r)^r,
	\qquad
	\kappa_r=r\log(1+b_r).
	\]
\end{proof}

	We are now ready to address the proof of the main convergence theorems.

\begin{proof}[Proof of Theorems~\ref{thm_main} and~\ref{thm_absorbing}]
	Let $\mathcal A=\emptyset$ in the nonabsorbing case, so that
	$\tau_{\mathcal A}=\infty$. Set $h=h_\Delta^\partial$ and
	$a=K_\Delta h$. For brevity, write
	$K(t)=K^\Delta(t)$ and $\tau_k=\tau_k^\Delta$.
	
	By Proposition~\ref{prop_embedding}, the embedding yields a
	coupling, and the absorbed path satisfies
	$\mX_t=\mX_{t\wedge\tau_{\mathcal A}}$.
	Therefore, by the definition of the $p$-Wasserstein
	distance~\eqref{eq:Wasserstein}, it suffices to bound
	\begin{equation}
		\mathcal E_T
		=\left\|\sup_{0\le t\le T}
		d(\mX_{t\wedge\tau_{\mathcal A}},
		\mX_{\tau_{K(t)}})\right\|_{L^p(\Prob_\bx)}.
	\end{equation}
	
	Choose a temporal H\"older exponent $\theta$ and $q>1$ such that
	\begin{equation}
		\label{eq_main_exponents}
		2\alpha<\theta<\frac12,\qquad
		pq\theta\le2,\qquad
		\theta<\frac12-\frac{q-1}{pq}.
	\end{equation}
	Such a choice is possible for every
	$\alpha<\frac14\wedge\frac1p$.
	Indeed, choose
	$\theta\in(2\alpha,\frac12\wedge\frac2p)$,
	then take $q>1$ sufficiently close to $1$.
	
	Fix $a_0>0$ and an integer $M>2(T+a_0)+1$;
	work first with $a\le a_0$. Define
	$E_M=\{\tau_{K(T)}\le M\}$, and
	\begin{equation}
		H_{\theta,M}=\sup_{0\le s<t\le M}
		\frac{d(\mX_s,\mX_t)}{|t-s|^\theta},
		\qquad
		D_T=\sup_{0\le t\le T}
		|(t\wedge\tau_{\mathcal A})-\tau_{K(t)}|.
	\end{equation}
	On $E_M$, both $t\wedge\tau_{\mathcal A}$
	and $\tau_{K(t)}$ lie in $[0,M]$.
	By H\"older's inequality, Lyapunov's inequality,
	Theorem~\ref{thm_KC}, and
	Proposition~\ref{prop_embedding_bounds},
	\begin{align}
		\label{eq_mainproof_bound1}
		&\left\|\mathbf1_{E_M}\sup_{t\le T}
		d(\mX_{t\wedge\tau_{\mathcal A}},
		\mX_{\tau_{K(t)}})\right\|_{L^p}\\
		&\quad\le
		\|H_{\theta,M}\|_{L^{pq/(q-1)}}
		\bigl(\Esp[D_T^{pq\theta}]\bigr)^{1/(pq)}\\
		&\quad\le C_{p,q,\theta,M,\bx}
		\bigl(24aT+9a^2\bigr)^{\theta/2}.
	\end{align}
	
	To control $E_M^c$ without assuming finite diameter,
	let $R_t=d(o,\mX_t)$ as in
	Lemma~\ref{lem_spatial_growth}.
	Since $K(t)\le K(T)$ for $t\le T$,
	on $\{\tau_{K(T)}\in[m,m+1)\}$, $m\ge M$,
	the triangle inequality gives
	\begin{equation}
		\sup_{t\le T}
		d(\mX_t,\mX_{\tau_{K(t)}})
		\le2\sup_{0\le u\le m+1}R_u.
	\end{equation}
	
	Choose conjugate exponents $r,r'>1$.
	By H\"older's inequality,
	Lemma~\ref{lem_spatial_growth}, and the exponential
	tail bound of Proposition~\ref{prop_probability_bound},
	with $\lambda=1/(2a)$,
	\begin{align}
		&\Esp\left[\mathbf1_{E_M^c}
		\Big(\sup_{t\le T}
		d(\mX_t,\mX_{\tau_{K(t)}})\Big)^p\right]\\
		&\quad\le2^p\sum_{m=M}^\infty
		\Prob\{\tau_{K(T)}\ge m\}^{1/r'}
		\Big(\Esp\sup_{0\le u\le m+1}R_u^{pr}\Big)^{1/r}\\
		&\quad\le C\sum_{m=M}^\infty
		\exp\left\{
		\frac{2(T+a)-m}{2ar'}
		+\frac{\kappa_{pr}(m+1)}r
		\right\}.
		\label{eq_tail_series_corrected}
	\end{align}
	
	For all sufficiently small $a$, the coefficient of $m$
	is at most $-1/(4ar')$.
	Factoring out the term at $m=M$ gives the bound
	\begin{equation}
		C\exp\left\{
		\frac{2(T+a)-M}{2ar'}
		+\frac{\kappa_{pr}(M+1)}r
		\right\}
		\sum_{j=0}^\infty e^{-j/(4ar')}.
	\end{equation}
	Since $M>2(T+a_0)+1$, the geometric series is
	uniformly bounded and the displayed expression is
	at most $Ce^{-c/a}$, with positive constants
	$C,c$ independent of the subdivision. Thus
	\begin{equation}
		\label{eq_mainproof_bound2}
		\left\|\mathbf1_{E_M^c}\sup_{t\le T}
		d(\mX_{t\wedge\tau_{\mathcal A}},
		\mX_{\tau_{K(t)}})\right\|_{L^p}
		\le Ce^{-c/(pa)}.
	\end{equation}
	
	Together,~\eqref{eq_mainproof_bound1} and
	\eqref{eq_mainproof_bound2} give
	\begin{equation}
		\mc W_T^p(\law(\mX),\law(\sX))
		\le C(a^{\theta/2}+e^{-c/(pa)})
		\le C_\alpha(h_\Delta^\partial)^\alpha.
	\end{equation}
	Indeed, $K_\Delta\le5/(\varepsilon\beta_*)$ by
	Proposition~\ref{prop_moments3}, with the convention
	$\beta_*=1$ when there are no nonabsorbing
	original vertices.
	
	For the absorbed process this is precisely
	\eqref{eq_Wasserstein_absorbing}; for the
	nonabsorbing process,
	$h_\Delta^\partial\le|\Delta|_{\mX}$ by
	\eqref{eq_boundary_old_comparison}, giving
	\eqref{eq_Wasserstein_bound}.
	All constants are independent of $\Delta$
	for sufficiently small thinness.
\end{proof}

	\subsection{Adaptation of subdivisions and optimality}
	\label{sec: adaptation} 
	
	We now analyze how adaptive subdivisions can improve the convergence rate in Theorem~\ref{thm_main}. 
	The one-dimensional case provides particularly clear insight into this phenomenon.
	
	Let $X$ be a diffusion on natural scale defined on an interval $J \subset \mathbb{R}$, and let $\bg$ be a grid on $J$, that is: an ordered locally finite countable family of elements of $J$. A grid is the one-dimensional counterpart of a subdivision.
	By \cite[Theorem~2.1]{anagnostakis2023general}, the one-dimensional
	STMCA satisfies
	\begin{equation}
		\label{eq_tuning_0}
		\mc W^{p}_{T}(\law(X),\law(\wt X^{\bg}))
		\le C(\alpha,p,T)\xabs{\bg}{X}^{\alpha},
		\qquad 0<\alpha<\tfrac14\wedge\tfrac1p,
	\end{equation}
	where $\abs{\bg}=\sup_{U\in\mc C(\bg)}|U|$ is the maximal step size
	and $\xabs{\bg}{X}=\sup_{U\in\mc C(\bg)}|U|m(U)$ is the thinness
	quantifier. A further bound by $C|\bg|^\alpha$ requires the
	additional mesh-dependent estimate $\xabs{\bg}{X}\le C_0|\bg|$;
	it does not follow from a lower speed-measure bound alone.

	A basic feature of the STMCA (as shown in \cite[Section~2.3]{anagnostakis2023general})
	is that one can take advantage of the thinness quantifier's structure to double the convergence rate. Indeed, using an adapted grid $\bg'$ that satisfies the geometric constraint: 
	\begin{equation}
		\label{eq:adapt_condition}
		\xabs{\bg\prime}{X}\le \abs{\bg\prime}^{2},
	\end{equation}
	we obtain the bound:
\begin{equation}
	\label{eq:adapted_bound}
	\mc W_T^p
	\bigl(\law(X),\law(\wt X^{\bg'})\bigr)
	\le C(\alpha,p,T)
	|\bg'|^{2\alpha},
	\qquad
	0<\alpha<\tfrac14\wedge\tfrac1p.
\end{equation}
	
	The adaptive condition~\eqref{eq:adapt_condition} naturally leads to grid refinement in regions where the speed measure $m$ is larger, providing higher resolution where the process spends more time.
	It is particularly useful for accurate simulation with low numerical cost of diffusions with speed measures that exhibit high anisotropy like sticky features or finite natural boundaries.
	The effectiveness is demonstrated numerically for the Cox--Ingersoll--Ross process in \cite[Section~5]{anagnostakis2023general}, showing significant improvement when the grid is adapted.
	
	For nonabsorbing diffusions, the same mechanism applies to diffusions on metric graphs. The absorbing case uses the modified functional from Definition~\ref{def_boundary_adapted}, as explained below.
	For a diffusion $\mX$ on a metric graph $\Gamma$, we define the subdivision step size:
	\begin{equation}
		\abs{\Delta} \coloneqq \sup_{U\in\mc{C}(\Delta)} \abs{U},
	\end{equation}
	where $\abs{U} = b-a $ if $U $ is the form~\eqref{eq: vertex_neighborhood0} and 
	$\abs{U} = \max_{e\in E(\bv)} u_e $ if $U $ is of the form~\eqref{eq: vertex_neighborhood}. 
	If, in addition to the hypotheses of Theorem~\ref{thm_main}, a family
	of subdivisions satisfies a uniform estimate
	\begin{equation}
	\xabs{\Delta}{\mX}\le C_0|\Delta|,
	\end{equation}
	then the theorem yields
	\begin{equation}
	\label{eq:adapted_bound_g1}
	\mc W^p_T(\law(\mX),\law(\sX))
	\le C C_0^{\alpha}|\Delta|^\alpha,
	\quad 0<\alpha<\tfrac14\wedge\tfrac1p,
	\end{equation}
	for sufficiently small $|\Delta|$. Such an upper estimate is not
	automatic for arbitrary speed measures.
	For an adapted family satisfying the stronger condition
	$\xabs{\Delta}{\mX}\le |\Delta|^2$, one gets
	\begin{equation}
	\label{eq:adapted_bound_g2}
	\mc W^p_T(\law(\mX),\law(\sX))
	\le C|\Delta|^{2\alpha},
	\quad 0<\alpha<\tfrac14\wedge\tfrac1p.
	\end{equation}
	This doubles the exponent under the stated mesh condition.

\subsection{Boundary-adapted graded grids}
\label{sec_boundary_grid}
The boundary-adapted thinness functional $\Theta_X$ allows finite grids near accessible absorbing endpoints, even when the original speed-mass thinness is infinite.
A family $\Delta_n$ with
$h_{\Delta_n}^\partial\le C_0|\Delta_n|^2$ therefore has,
under the hypotheses of Theorem~\ref{thm_absorbing},
\begin{equation}\label{eq_adapted_absorbing_mesh}
\mc W_T^p(\law(\mX),\law(\wt\mX^{\Delta_n}))
\le C |\Delta_n|^{2\alpha},\qquad
0<\alpha<\tfrac14\wedge\tfrac1p.
\end{equation}
Here the spatial step $|\Delta_n|$ is measured in the original
geodesic metric, not a modified directed metric.

\begin{example}[Power-law absorbing boundary]
	\label{ex_absorbing_power_law}
	On $[0,L]$ take natural scale $s(x)=x$ and speed
	$m(\rd x)=c x^{-\gamma}\rd x$ for
	$1\le\gamma<2$, $c>0$.
	The endpoint $0$ is an exit boundary, whereas
	$L$ is regular. We consider the diffusion
	absorbed upon reaching either endpoint.
	
	Both endpoints are included in the grid. Set
	\begin{equation}\label{eq_boundary_power_grid}
		x_k=L\left(\frac{k}{n}\right)^r,\qquad
		k=0,\ldots,n,\qquad
		r=\frac{2}{2-\gamma}.
	\end{equation}
	
	For $n\ge2$, the first active cell is
	$U_{x_1}=(x_0,x_2)=(0,x_2)$.
	Direct integration gives
	\begin{equation}\label{eq_boundary_first_cell}
		\Theta_X(U_{x_1})
		=\frac{2c}{(2-\gamma)(3-\gamma)}
		x_2^{2-\gamma}
		=O(n^{-2}).
	\end{equation}
	
	For $2\le k\le n-1$, put
	$a=x_{k-1}$ and $b=x_{k+1}$.
	The diagonal Green bound implies
	\begin{equation}
		\Theta_X(a,b)
		\le\frac{c}{2}(b-a)^2a^{-\gamma}
		\le Cn^{-2},
	\end{equation}
	since
	$b-a\le 2Lr(k+1)^{r-1}n^{-r}$
	and $r(2-\gamma)=2$.
	
	Therefore,
	\begin{equation}
		h_{\Delta_n}^\partial=O(n^{-2}),
		\qquad
		|\Delta_n|
		=L\left[1-\left(1-\frac2n\right)^r\right]
		\asymp n^{-1}.
	\end{equation}
	Consequently,
	\begin{equation}
		h_{\Delta_n}^\partial\le C|\Delta_n|^2,
	\end{equation}
	so the family satisfies the quadratic
	boundary-adapted thinness condition.
	
	The original thinness is infinite, since
	$m((0,x_2))=\infty$.
	
	In particular, for $\gamma=1$ and
	$c=\sigma^{-2}$, which corresponds to
	$\rd X_t=\sigma\sqrt{X_t}\,\rd B_t$
	up to its first hit of $\{0,L\}$,
	\begin{equation}
		\Theta_X(0,x_2)=\frac{x_2}{\sigma^2},
		\qquad
		x_k=L(k/n)^2.
	\end{equation}
	
	For every $1\le\gamma<2$, the corresponding
	natural-scale diffusion has the stochastic
	differential realization
	$\rd X_t=c^{-1/2}X_t^{\gamma/2}\rd B_t$
	until the first hit of $\{0,L\}$ and is
	constant thereafter. Its volatility is
	bounded on $[0,L]$; the
	Burkholder--Davis--Gundy inequality therefore
	verifies Condition~\ref{cond_Lp}.
	Thus, the improved convergence bound
	\eqref{eq_adapted_absorbing_mesh}
	applies to this entire family.
\end{example}

\section{Numerical experiments}\label{sec_numexp}
We examine the STMCA in three settings: Walsh Brownian
motion with a sticky or non-sticky central vertex, a
diffusion with heterogeneous edge dynamics, and a
square-root diffusion absorbed at zero.

For the two explicitly solvable benchmarks, we compare
the terminal distributions at $T=1$ using the
$p$-Wasserstein distances
\begin{equation}\label{eq:num_marginal}
	W_{p,T}(h)
	=W_p\bigl(\law(\mX_T),\law(\sX_T)\bigr),
	\qquad p=1,2.
\end{equation}
Here the transportation cost is induced by the graph
geodesic distance and $h$ denotes the spatial mesh
parameter.

Since evaluation at $T$ is $1$-Lipschitz with
respect to the uniform path metric,
\[
W_{p,T}(h)
\leq \mc W_T^p\bigl(\law(\mX),\law(\sX)\bigr).
\]
Thus, convergence of the terminal marginals
alone does not establish the path-space
convergence rate in Theorem~\ref{thm_main}.

For the Walsh benchmark, the exact diffusion and the
STMCA have the same distribution of the outgoing ray,
independent of their radial coordinates. Consequently,
their graph Wasserstein distances coincide with the
Wasserstein distances between their radial distributions.

Let $F_T$ and $F_T^h$ denote the cumulative distribution
functions of the radial coordinates of the exact
diffusion and its STMCA, respectively. By the
one-dimensional quantile representation of Wasserstein
distances (see~\cite[Proposition~2.17]{santambrogio2015}),
we obtain
\begin{equation}\label{eq:num_wasserstein_quantiles}
	W_{p,T}(h)^p
	=
	\int_0^1
	\left|
	F_T^{-1}(u)-(F_T^h)^{-1}(u)
	\right|^p
	\,\rd u,
	\qquad p=1,2.
\end{equation}

For the Walsh benchmark, we compute the STMCA
terminal distribution by deterministic probability
propagation based on the semi-Markov renewal
equations (see~\cite[Chapter~4, Eq.~(10.2)]{janssen2006applied}).

Indeed, let
\[
P^\Delta_{\bx,\by}(t)
\coloneqq
\Prob_{\bx}(\sX_t=\by),
\qquad \bx,\by\in V_\Delta.
\]
Conditioning on the first transition gives
\begin{align}
	P^\Delta_{\bx,\by}(t)
	={}&
	\indic{\bx=\by}
	\sum_{\bz}
	\fp_{\bx,\bz}
	\indic{t<\ft_{\bx,\bz}}\\
	&+
	\sum_{\bz}
	\fp_{\bx,\bz}
	\indic{t\ge\ft_{\bx,\bz}}
	P^\Delta_{\bz,\by}
	(t-\ft_{\bx,\bz}),
\end{align}
where the sums run over the possible destinations
$\bz$ of $\bx$.

In this symmetric Walsh setting, we implement
the renewal equation through a forward recursion
for the radial probability masses, incorporating
the additional holding time at the central vertex.
Since the holding times are integer multiples
of $h^2$ for the parameters considered, the
recursion operates on a time grid with step
$h^2$, without rounding the holding times.

We then compare the resulting terminal distribution
with the exact reference law using
\eqref{eq:num_wasserstein_quantiles}.

In the absorbing example, the same one-dimensional
quantile formula is applied to an empirical STMCA
terminal distribution and the exact Poisson--Gamma
mixture. In both benchmarks, any atom at the
central vertex or absorbing boundary is retained
in the calculation.

\subsection{Walsh Brownian motion with a sticky vertex}
Consider a three-ray star with $s_e(x)=x$, $m_e(\rd x)=\rd x$, and spinning probabilities $(\beta_1,\beta_2,\beta_3)=(0.15,0.30,0.55)$. At the central vertex $\bv$, impose
\[
\sum_{e=1}^3\beta_e f'_e(0)=\rho\,\Lop f(\bv),
\qquad \rho\in\{0,0.5,2\}.
\]

We start the diffusion from $\mX_0=\bv$.
Let $R_T=d(\bv,\mX_T)$ denote its radial
coordinate at time $T$, and let $I_T$ denote
the occupied edge whenever $R_T>0$.

In the non-sticky case, $R_T$ has the law
of $|B_T|$
\cite{walsh1978diffusion,barlow1989onwalsh}.
For $\rho>0$, the radial distribution
has an atom at zero of mass
\[
\Prob(R_T=0)=\operatorname{erfcx}\!\left(\frac{\sqrt{2T}}{\rho}\right)
\]
and a density on $(0,\infty)$ given by
\[
q_\rho(T,r)=\frac{2}{\rho}\exp\!\left(-\frac{r^2}{2T}\right)
\operatorname{erfcx}\!\left(\frac{r+2T/\rho}{\sqrt{2T}}\right).
\]
These expressions follow from the sticky star-graph resolvent \cite[Lemma~3, Eq.~(3.5)]{bonaccorsi_dovidio_2024}, with convention $\rho=c/b$; see also \cite{anagnostakis2025walsh,peskir_2022}. 
The uniform-grid STMCA uses mean holding times $h^2$ away from the vertex and $h^2+\rho h$ at the vertex.

\begin{figure}[htbp]
	\centering
	\includegraphics[width=.99\linewidth]{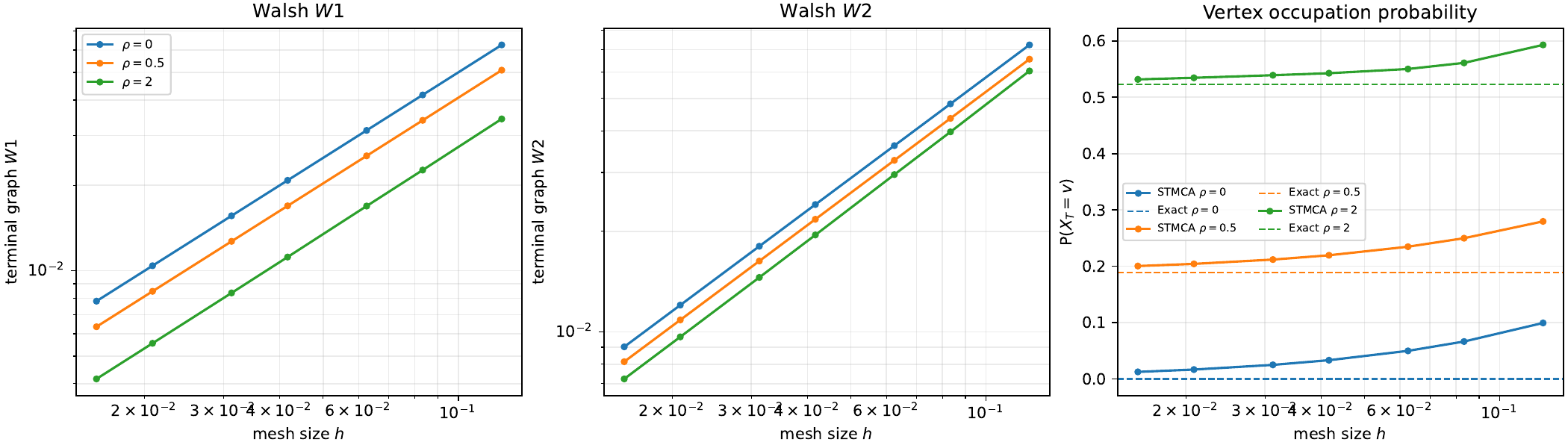}
	\includegraphics[width=.99\linewidth]{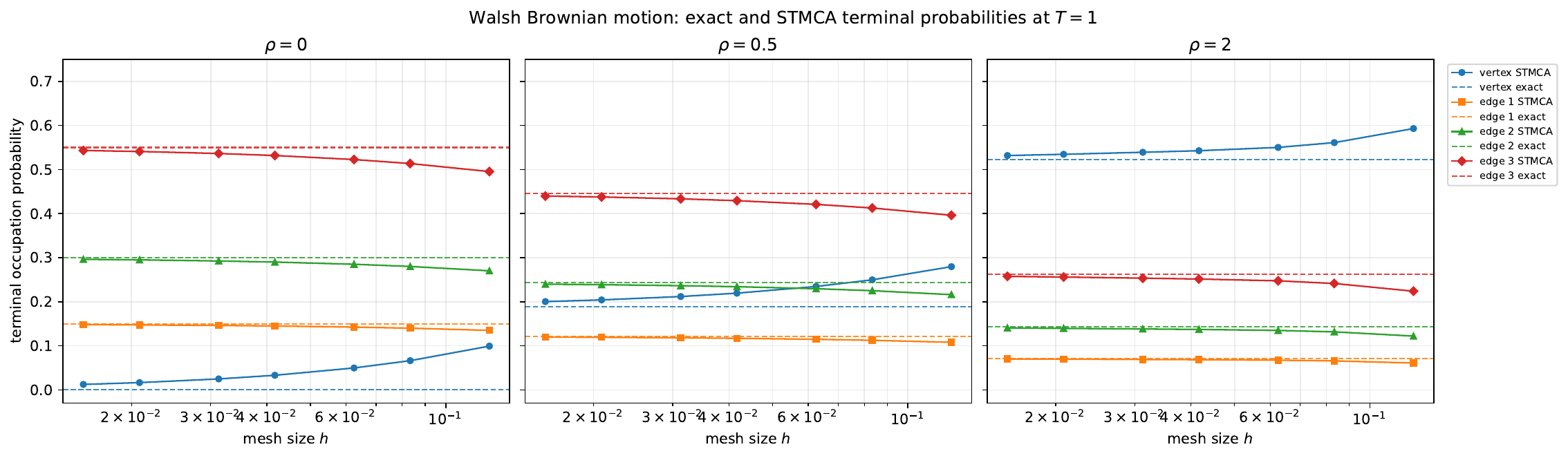}
	\caption{    Walsh Brownian motion at $T=1$ for
		$\rho=0,0.5,2$.
		Top: terminal graph $W_1$ and $W_2$
		distances (log--log axes) and
		central-vertex occupation probabilities.
		Bottom: terminal occupation probabilities
		of the three rays.
		The STMCA results are compared with the
		corresponding exact probabilities.}
	\label{fig:num_walsh_convergence}
\end{figure}

For all three stickiness values, the terminal
errors in both Wasserstein orders decrease
approximately linearly with $h$.
For example, when $\rho=0.5$, the $W_1$
error falls from $0.05096$ at $h=1/8$
to $0.00635$ at $h=1/64$, and $W_2$
from $0.06552$ to $0.00814$.

Since the outgoing ray is selected according
to the spinning probabilities, independently
of the positive radial coordinate, the exact
terminal ray-occupation probabilities satisfy
\[
\Prob(I_T=j,\ R_T>0)
=
\beta_j\bigl(1-\Prob(R_T=0)\bigr),
\qquad j=1,2,3.
\]

\subsection{Heterogeneous edge dynamics}

We next consider a diffusion on a three-ray star graph,
with natural scale $s_e(x)=x$ on each edge and
speed measures
\begin{equation}\label{eq:num_hetero_speeds}
	m_1(\rd x)=\frac{\rd x}{2(\varepsilon+x)^2},
	\qquad
	m_2(\rd x)=\frac{\rd x}{2(\varepsilon+x)},
	\qquad
	m_3(\rd x)=\frac{\rd x}{2}.
\end{equation}
The three edges exhibit shifted geometric-Brownian-type,
driftless shifted square-root (CIR-type), and Brownian
dynamics, respectively.

At the central vertex $\bv$, we impose the
Wentzell-type transmission condition
\[
\frac13\sum_{e=1}^{3}f'_e(0)
=\rho\,\Lop f(\bv),
\qquad \rho\in\{0,0.5\},
\]
corresponding to equal spinning probabilities
$\beta_e=1/3$ for $e=1,2,3$, with either
non-sticky or sticky vertex behavior.

In the STMCA, transition probabilities are determined
by the scale functions, while destination-conditional
holding times are computed from the Green kernel and
the corresponding speed measures. We take
$\varepsilon=0.5$, $h=0.025$, and $T=1$.
For numerical implementation, we evaluate the exact
Green-function representations of the
destination-conditional holding times using
Gauss--Legendre quadrature. At the central vertex,
we use symmetric star-neighborhoods of radius
$u=h^2$, with transition probabilities
$\beta_j=1/3$, and compute the holding times from
Proposition~\ref{prop_moments2}. In particular,
the asymptotic contribution $\rho u$ from
Proposition~\ref{prop_asymptotic} is retained
together with the edgewise integral corrections,
rather than used as a standalone approximation.

Figure~\ref{fig:num_heterogeneous_tails} shows
sample STMCA trajectories for $\rho=0$ and
$\rho=0.5$.
To compare the heterogeneous terminal tails, let
$R_T=d(\bv,\mX_T)$ and
$\widetilde R_T=d(\bv,\sX_T)$ denote the
radial coordinates of the exact diffusion and
its STMCA, respectively. Let $J_T$ and
$\widetilde J_T$ denote their occupied edges
whenever the corresponding radial coordinate is
positive. Define the edge-conditional survival
functions
\[
S_j(r)=\Prob(R_T>r\mid J_T=j),
\qquad
\widetilde S_j(r)=\Prob(\widetilde R_T>r\mid\widetilde J_T=j),
\qquad r\geq0.
\]

For the heterogeneous STMCA with $\rho=0$,
we estimate $\widetilde S_j$ by Monte Carlo
simulation using $2\times10^5$ independent
trajectories. Denoting the empirical estimates
by $\widehat S_j$, we plot the ratios
$\widehat S_1/\widehat S_2$ and
$\widehat S_3/\widehat S_2$, taking the
CIR-type edge as reference.

For qualitative comparison, we also display ratios
constructed from the explicit terminal distributions
of the corresponding standalone one-dimensional
diffusions, each started at radius $0.1$. Their
distributions are shifted lognormal, shifted
squared-Bessel, and Gaussian, respectively.
These reference processes do not incorporate
switching at the central vertex and therefore
do not provide exact benchmarks for the
graph diffusion.

\begin{figure}[htbp]
	\centering
	\includegraphics[width=.98\linewidth]
	{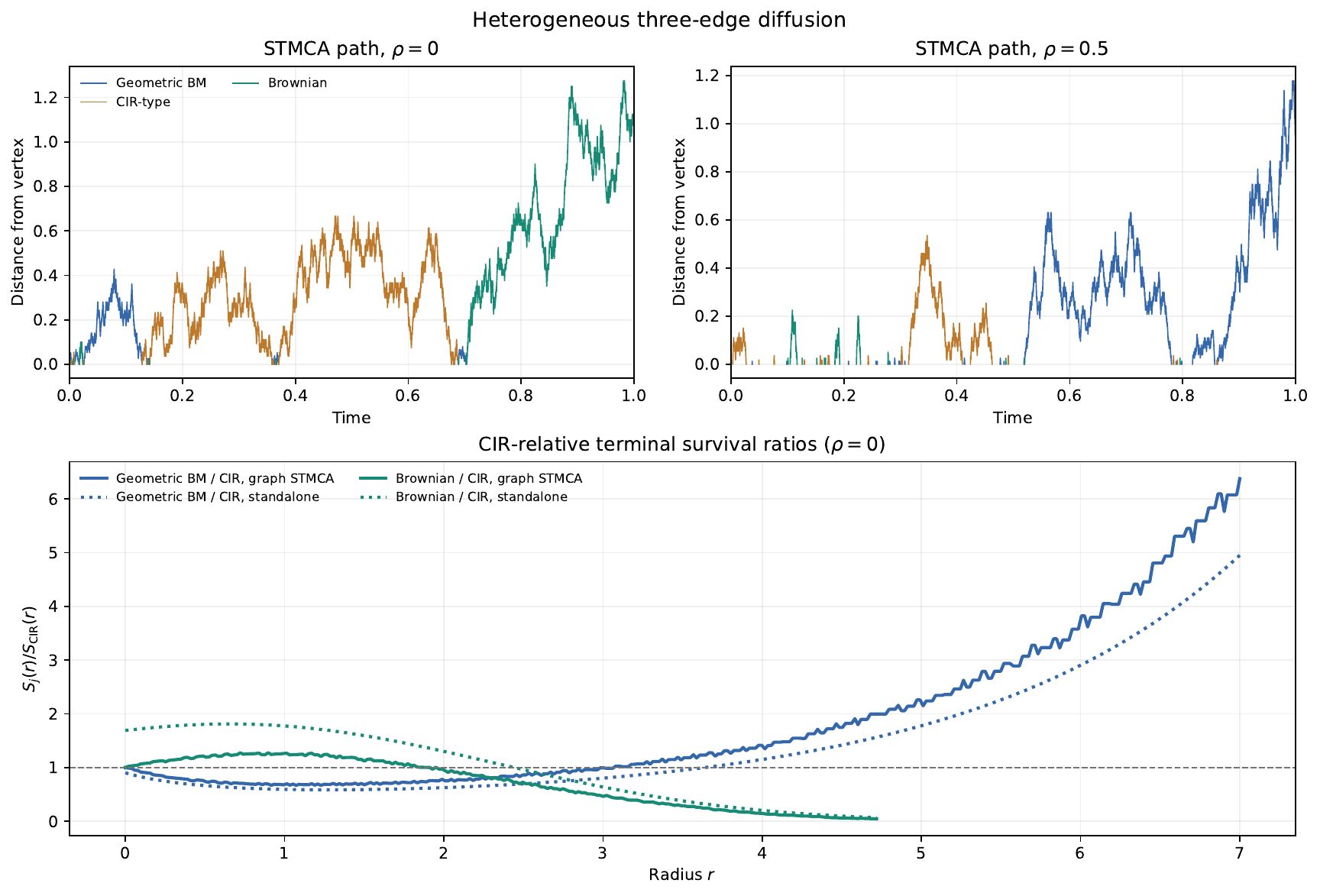}
	\caption{
		Heterogeneous diffusion on a three-ray star graph
		with $\varepsilon=0.5$ and $T=1$.
		Top: sample STMCA trajectories for non-sticky
		($\rho=0$) and sticky ($\rho=0.5$) vertex conditions.
		Bottom: ratios of empirical edge-conditional
		terminal survival functions for the non-sticky
		model ($2\times10^5$ independent paths,
		$h=0.025$).
		The CIR-type edge is used as the reference.
		Dotted curves represent ratios for standalone
		one-dimensional edge diffusions started at
		radius $0.1$; they are not exact graph
		diffusion benchmarks.
		Ratios are displayed only where each
		empirical tail contains at least $40$
		observations.
	}
	\label{fig:num_heterogeneous_tails}
\end{figure}

The numerical results reveal a clear ordering of the edge-conditional terminal tails: the shifted geometric-Brownian-type edge exhibits the heaviest tail, followed by the CIR-type edge and then the Brownian edge. This ordering holds over the range resolved by the simulations and illustrates the influence of heterogeneous edge dynamics on the terminal distribution. The experiment does not assess time-marginal or path-space Wasserstein convergence rates.

\subsection{An absorbing boundary}

Finally, we consider the square-root diffusion
\begin{equation}\label{eq:num_absorbing_sde}
	\rd X_t=\sqrt{X_t}\,\rd B_t,
	\qquad X_0=0.4,
\end{equation}
with absorption at zero. This diffusion is on
natural scale, with $s(x)=x$ and
$m(\rd x)=\rd x/x$.

The rescaled process $Z_t=4X_t$ satisfies
\[
\rd Z_t=2\sqrt{Z_t}\,\rd B_t,
\qquad Z_0=4X_0,
\]
and is therefore a squared Bessel process of
dimension zero
(see~\cite[Definition~XI.1.1]{RevYor}).
Its transition law is characterized by the
zero-dimensional squared Bessel semigroup
and the corresponding scaling property
(see~\cite[Corollary~XI.1.4
and Proposition~XI.1.6]{RevYor}).
Equivalently,
\[
\frac{4X_T}{T}
\sim
\chi'^2_0\left(\frac{4X_0}{T}\right),
\]
where $\chi'^2_0(\lambda)$ denotes the
noncentral chi-squared distribution with zero
degrees of freedom and noncentrality parameter
$\lambda$; see also~\cite{chari_1990} for
the connection between squared Bessel diffusions
and noncentral chi-squared distributions.

By the Poisson-mixture representation of
$\chi'^2_0(\lambda)$
(see~\cite[Section~1]{siegel_1979}),
the exact terminal law of $X$ admits the
following representation. Let
\[
N_T\sim\operatorname{Poisson}
\left(\frac{2X_0}{T}\right).
\]
Then
\[
X_T\mid\{N_T=n\}\sim
\begin{cases}
	\delta_0, & n=0,\\
	\operatorname{Gamma}(n,T/2), & n\ge1,
\end{cases}
\]
where the Gamma distribution is parametrized
by its shape and scale.

In particular, since $X_0=0.4$, the exact
probability of absorption by time $T=1$ is
\[
\begin{aligned}
	\Prob(X_1=0)
	&=\Prob(N_1=0)\\
	&=\exp(-2X_0)
	=e^{-0.8}
	\simeq 0.449329.
\end{aligned}
\]

For the cell $U=(0,2h)$ centered at $h$,
let $\tau_U$ denote its exit time. Since
$s(x)=x$ and $m(\rd y)=\rd y/y$, the exit
probabilities are
\[
p_0(y)=\frac{2h-y}{2h},
\qquad
p_{2h}(y)=\frac{y}{2h}.
\]
Moreover, the Green kernel satisfies
\[
G(h,y)=
\begin{cases}
	y/2, & 0<y\le h,\\
	(2h-y)/2, & h<y<2h.
\end{cases}
\]
Thus, Lemma~\ref{lem_boundary_Green} gives
\begin{align*}
	\Esp_h[\tau_U\mid X_{\tau_U}=0]
	&=\frac1h\left(
	\int_0^h(2h-y)\,\rd y
	+\int_h^{2h}\frac{(2h-y)^2}{y}\,\rd y
	\right)\\
	&=(4\log 2-1)h,\\
	\Esp_h[\tau_U\mid X_{\tau_U}=2h]
	&=\frac2h\int_0^{2h}G(h,y)\,\rd y
	=h.
\end{align*}
Consequently, the STMCA assigns equal
transition probabilities $1/2$ to the two
destinations, but different conditional
holding times. After reaching zero, the
STMCA remains absorbed there.

Using $10^5$ independent trajectories per mesh,
we estimate the terminal $W_1$ and $W_2$
distances between the STMCA and the exact
diffusion. As the mesh parameter decreases
from $h=0.1$ to $h=0.02$, the estimated
$(W_1,W_2)$ distances decrease from
$(0.01575,0.02626)$ to $(0.00384,0.00813)$.
At the finest mesh, the simulated absorption
probability is $0.44750$, consistent with
the exact value at the Monte Carlo sampling
resolution.

Figure~\ref{fig:num_absorbing_convergence}
presents the terminal Wasserstein distances
and compares the simulated absorption
probabilities with their exact values.

\begin{figure}[htbp]
	\centering
	\includegraphics[width=.96\linewidth]
	{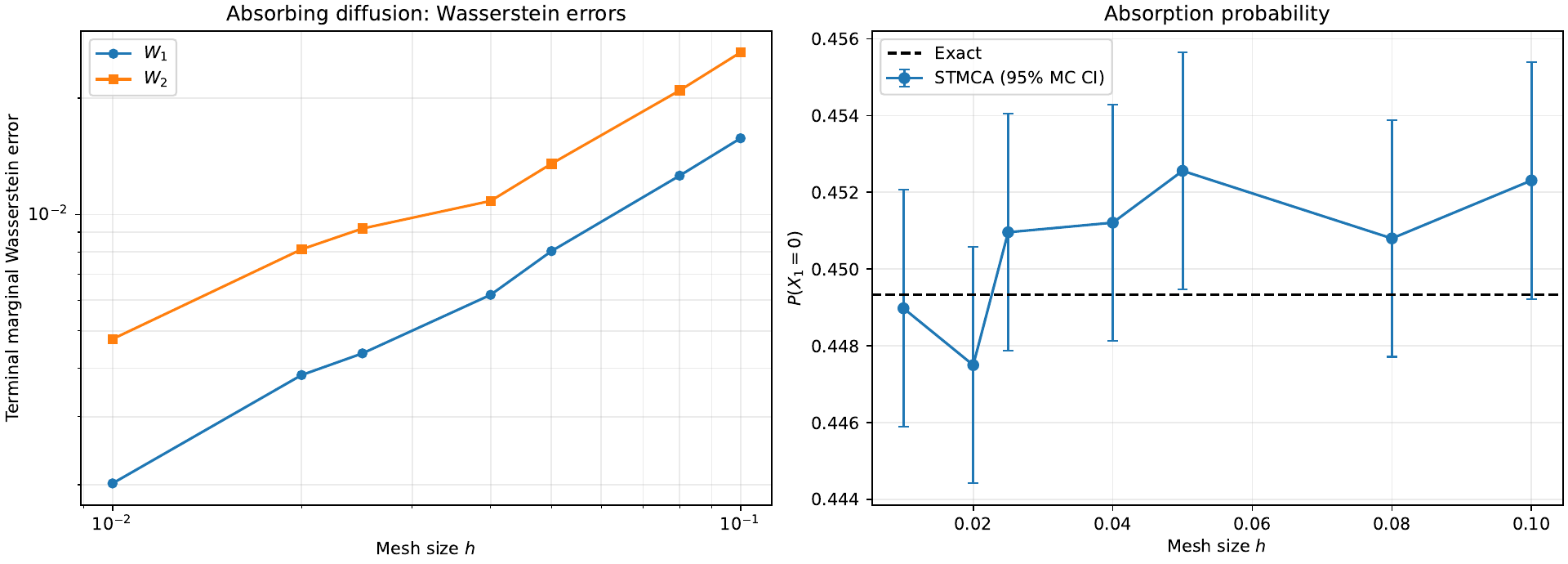}
	\caption{
		Square-root diffusion absorbed at zero:
		estimated terminal $W_1$ and $W_2$
		distances to the exact Poisson--Gamma
		law, and the probability of absorption
		by $T=1$. The STMCA uses
		destination-conditional holding times;
		the dashed line represents the exact
		absorption probability $e^{-0.8}$,
		and the error bars indicate $95\%$
		binomial Monte Carlo confidence intervals.
	}
	\label{fig:num_absorbing_convergence}
\end{figure}

The Walsh and absorbing benchmarks exhibit
decreasing terminal Wasserstein distances
to their respective exact distributions,
while the heterogeneous example illustrates
the influence of distinct edge dynamics
and vertex conditions on terminal tail behavior.
In particular, the absorbing experiment
illustrates the direction-sensitive boundary
treatment discussed in
Remark~\ref{rmk: absorption}.
However, these numerical comparisons do not
independently establish the convergence
results of Theorems~\ref{thm_main}
and~\ref{thm_absorbing}, since they do not
measure path-space Wasserstein distances.

\section{Conclusion}
\label{sec_conclusion}

We have developed the Space-Time Markov Chain Approximation for general diffusions on finite metric graphs, which are essentially one-dimensional diffusion dynamics coupled through transmission conditions at the vertices of a finite metric graph. 
The scheme matches the diffusion's exit probabilities and destination-conditional
mean exit times on each subdivision cell. 
Explicit local formulas make the resulting semi-Markov approximation implementable, while
the embedding construction and exit-time moment bounds yield
quantitative path-space Wasserstein convergence for non-explosive diffusions on
natural scale on each edge under the stated regularity and mesh
assumptions.

For accessible absorbing endpoints, a boundary-adapted thinness
based on the diagonal killed Green kernel accommodates cells whose
unweighted speed mass is infinite. The graded-grid example shows
how this construction can yield quantitative rates in terms of the
spatial mesh size. The numerical benchmarks illustrate terminal-law
accuracy for sticky, heterogeneous, and absorbing dynamics; they
do not directly measure path-space Wasserstein distances. Extending
the convergence analysis to non-NSE diffusions and developing
adaptive-grid selection criteria remain natural directions for
further work.

\appendix
	
	\section{Dirichlet problems in the vicinity of a vertex} 
	\label{app_dirichlet}

	Here we prove existence, uniqueness, and an analytic expression for the Dirichlet problem in the neighborhood of a vertex. 
	This improves the existing result \cite[Proposition~B.1]{anagnostakis2025walsh} in two aspects: (i) it considers neighborhoods of
	$\bv $ that are not necessarily symmetric and (ii) it allows for non-trivial boundary conditions. 
	
	\begin{proposition}
		\label{prop_Dirichlet}
Let $\Gamma=(V,E)$ be a star-graph (see
Definition~\ref{def_stargraph}) with central vertex
$\bv$ and edge lengths $(l_e)_{e\in E}$.
Consider a relatively compact star-shaped open
neighborhood $U$ of $\bv$ of the form
\begin{equation}
	\label{eq_form_neighborhood}
	U\coloneqq
	\bigcup_{e\in E}
	\{(e,x):0\le x<u_e\},
	\qquad 0<u_e<l_e.
\end{equation}
		Suppose each $s_e$ extends continuously to $[0,u_e]$, is normalized by $s_e(0)=0$, and has finite strictly positive right derivative $0<s'_e(0+)<\infty$. Assume the speeds have finite mass on $(0,u_e)$ and let $g$ be bounded and continuous on the closed star neighborhood $\overline U$ (with a common value $g(\bv)$ at the center). For boundary data $(a_e)_{e\in E}\in\IR^E$, consider the Dirichlet problem: 
		\begin{equation}
			\label{eq_boundary_value_problem}
			\begin{cases}
				\frac{1}{2}\D_{m_{e}} \D_{s_e} f(e,x) = g(e,x), &\forall\, x\in (0,u_e),\; \forall\, e\in E,\\
				\rho \Lop_e f(e,0) = \sum_{i\in E} \beta_{i} f'(i,0),
				&\forall\, e\in E,\\
				f(e,u_e) = a_e, &\forall\, e \in E, \\
				f(e,0) = f(e',0), &\forall\, e,e'\in E,
			\end{cases}
		\end{equation}
		where $f'(e,0)$ denotes the right derivative
		of $f$ along edge $e$ at the central vertex.
		\begin{enuroman}
			\item \label{item_existence} A solution to~\eqref{eq_boundary_value_problem} is
			the function $f$ on $\overline U$ defined as
			\begin{equation}
				\label{eq_solution_Dirichlet}
				f(e,x) \coloneqq \frac{s_e(u_e)A_{e}(x) - s_e(x)A_{e}(u_e)}{s_e(u_e)}
				+ C \frac{s_{e}(u_e)-s_e(x)}{s_e(u_e)}
				+ a_e \frac{s_e(x)}{s_e(u_e)},
			\end{equation}
			where $(A_e;\, e\in E) $ and $C $ are defined as
			\begin{equation}
				\label{eq_constants_Dirichlet}
				\begin{dcases}
					A_e(x) \coloneqq 2 \int_{(0,x)} \int_{(0,y)} g_e(\zeta) \,m_{e}(\rd \zeta)  \,s_{e}(\rd y),
					\qquad \text{for all } (e,x)\in U, \\
					C \coloneqq - \braces{\sum_{e\in E}\beta_e \frac{s'_e(0)}{s_e(u_e)}}^{-1} 
					\braces{ \rho g(\bv) 
						+ \sum_{e\in E} \braces{ A_e(u_e) - a_e}
						\beta_e \frac{ s'_e(0)}{s_e(u_e)}}.
				\end{dcases}
			\end{equation}
			\item \label{item_uniqueness}
			The problem~\eqref{eq_boundary_value_problem}
			has a unique solution among functions
			$f\in C(\overline U)$ whose restrictions
			$f_e$ belong to $C^{\Lop_e}([0,u_e])$
			for every $e\in E$, where
			\[
			C^{\Lop_e}([0,u_e])
			\coloneqq
			\left\{
			h\in C([0,u_e]):
			\Lop_e h\in C([0,u_e])
			\right\}.
			\]
		\end{enuroman}
	\end{proposition}
	
	\begin{proof}
		Let us assume the general form of the solution:
		\begin{align}
			f_e(x) &\coloneqq A_e(x) + B_e s_e(x) + C_e,\quad \forall\, x \in [0,u_e), \quad \forall\, e\in E.
		\end{align}
		
		We observe that
		\begin{equation}
			\Lop_e f_e(x) = \frac{1}{2} \D_{m_e}\D_{s_e} f_e(x)
			= g_e(x), \quad \forall\, x\in (0,u_e),\quad\forall\, e\in E.
		\end{equation}
		The constants $B_e $ and $C_e $ are determined so that
		\begin{xenumerate}{r}
			\item \label{item_junctionB} $f(e,0)=f(e',0) $, for all $e,e' \in E $, 
			\item \label{item_boundary} $f(e,u_e)=a_e $, for all $e\in E $. 
			\item \label{item_junctionA} $\rho \Lop_e f_e(0) = \sum_{e' \in E} \beta_{e'} f'_{e'}(0) $,
			for all $e\in E $. 
		\end{xenumerate}
		By the stipulated finite positive right derivative of $s_e$ and the local finiteness of $m_e$ at the central vertex, $A_e(x)=o(x)$ as $x\downarrow0$; hence $A'_e(0+)=0$.
		The resulting endpoint values and derivatives are:
		\begin{align}
			f_e(0)&= B_e s_e(0) + C_e,\\
			f'_e(0)&= B_e s'_e(0),\\
			f_e(u_e)&= A_e(u_e)
			+ B_e s_e(u_e) + C_e, \quad \forall\, e\in E
		\end{align}
		
		Recall that we assume the convention $s_e(0) =0 $, for all $e\in E $.
		Therefore, relation~\ref{item_junctionB} reads:
		\begin{equation}
			C \coloneqq C_{e} = C_{e'}, \quad \text{for all } e,e'\in E.  
		\end{equation}
		Relation~\ref{item_boundary} reads:
		\begin{align}
			A_e(u_e)
			+ B_e s_e(u_e) + C
			= a_e 
			\quad \Rightarrow \quad  B_e = \frac{a_e - A_e(u_e) - C }
			{s_e(u_e)} ,\quad \text{for all } e\in E. 
		\end{align}
		Relation~\ref{item_junctionA} reads:
		\begin{align}
			& \rho g(\bv) = \sum_{e\in E} \beta_e B_e s'_e(0)
			= \sum_{e\in E} \frac{a_e - A_e(u_e) - C }
			{s_e(u_e)} \beta_e  s'_e(0)\\
			\Rightarrow \quad &
			C = \braces{\sum_{e\in E}\beta_e \frac{s'_e(0)}{s_e(u_e)}}^{-1} 
			\braces{- \rho g(\bv) 
				+ \sum_{e\in E} \braces{a_e - A_e(u_e)}
				\beta_e \frac{ s'_e(0)}{s_e(u_e)}}.
		\end{align}
		This proves~\ref{item_existence}.
		
		The uniqueness (Assertion~\ref{item_uniqueness}) is a direct consequence of the maximum principle. Let us recall the argument, which is detailed in the closing of \cite[Proposition~B.1]{anagnostakis2025walsh}.
		
		We first prove uniqueness for the homogeneous problem: \( g \equiv 0 \) and $a_e = 0 $, for all $e\in E$ . The general solution to the homogeneous problem is given by (see \cite[Section~9]{feller1955secondorder}):
		\begin{equation}
			f(e, x) = B_e s_e(x) + C, \quad \forall\, e \in E, \quad \forall\, x > 0,
		\end{equation}
		where \( B_e, C \in \mathbb{R} \) for all \( e \in E \). The gluing condition for the homogeneous problem is:
		\begin{equation}
			\sum_{e \in E} \beta_e f'_e(0) = 0.
		\end{equation}
		
		The solution \( f \) is necessarily monotonic on each edge: increasing on \( e \) if \( B_e \geq 0 \) and decreasing on \( e \) if \( B_e \leq 0 \). This implies that \( f \) attains its maxima and minima either at \( \bv \) or on \( \partial U \).
		
		We now analyze the location of the maximum and minimum:
		\begin{itemize}
			\item Suppose the maximum is attained at $\bv $. Then, $B_e\le 0 $, for all $e\in E $.
			If $B_e=0 $, for all $e\in E $, then $f $ is constant on $U $.
			Otherwise, $\sum_{e\in E} \beta_e f'_e(0)< 0 $, which contradicts the gluing condition.
			\item Suppose the minimum is attained at $\bv $. Then  $B_e\ge 0 $, for all $e\in E $. 
			If $B_e=0 $, for all $e\in E $, then $f$ is constant on $U$.
			Otherwise, $\sum_{e\in E} \beta_e f'_e(0)> 0 $, which again contradicts the gluing condition.
		\end{itemize}
		Therefore, maxima and minima of \( f \) are necessarily located on \( \partial U \). 
		
		To prove uniqueness, let \( f_1 \) and \( f_2 \) be two solutions to the boundary value problem \eqref{eq_boundary_value_problem}. Then \( f = f_1 - f_2 \) satisfies the homogeneous problem with trivial boundary conditions. Since maxima and minima are attained at the boundary where $f$ function vanishes, it follows that $f$ is trivially null and that \( f_1 = f_2 \). This completes the proof of uniqueness.   
	\end{proof}
	
	We now develop alternative expressions for two particular cases of the above result that we use in this paper: the homogeneous case $g\equiv 0 $ and the case of trivial boundary conditions ($a_e = 0 $, for all $e\in E $).
	For notational convenience, let 
	\begin{equation}
		d_e \coloneqq \frac{\beta_e s'_e(0)}{s_e(u_e)}, 
		\quad \text{for all } e\in E.
	\end{equation}
	
	\begin{remark}
		\label{rmk_repr_a=0}
		If $g \equiv 0 $,
		then the unique solution reads
		\begin{equation}
			f(e,x) = \frac{s_{e}(u_e)-s_e(x)}{s_e(u_e)} \braces{\sum_{i\in E} d_i}^{-1} 
			\braces{ \sum_{i\in E} a_i d_i }
			+ a_e \frac{s_e(x)}{s_e(u_e)}. 
		\end{equation}
	\end{remark}
	
	\begin{remark}
		\label{rmk_repr_alt1}
		For zero boundary data $a_e=0$, the solution of
		$\Lop f=g$ can be written
		\begin{align}
		f(e,x)={}&-2\int_{(0,u_e)}G^e_{0,u_e}(x,y)g(e,y)\,m_e(\rd y)\\
		&-\frac{s_e(u_e)-s_e(x)}{s_e(u_e)}\frac{1}{\sum_i d_i}
		\bigg(\rho g(\bv)+2\sum_i d_i\int_{(0,u_i)}
		(s_i(u_i)-s_i(y))g(i,y)\,m_i(\rd y)\bigg).
		\end{align}
	\end{remark}

	\begin{remark}
		\label{rmk_repr_alt2}
		Equivalently, for the killed occupation problem $\Lop f=-g$
		with zero boundary values,
		\begin{equation}
		\label{eq_solution_Dirichlet0}
		f(e,x)=2\int_{(0,u_e)}G^e_{0,u_e}(x,y)g(e,y)\,m_e(\rd y)
		+C''\frac{s_e(u_e)-s_e(x)}{s_e(u_e)},
		\end{equation}
		where
		\begin{equation}
		\label{eq_constants_Dirichlet0}
		C''=\frac{\rho g(\bv)
		+2\sum_i d_i\int_{(0,u_i)}(s_i(u_i)-s_i(y))g(i,y)\,m_i(\rd y)}{\sum_i d_i},
		\qquad d_i=\frac{\beta_i s'_i(0)}{s_i(u_i)}.
		\end{equation}
	\end{remark}

	\section{A Green function for a diffusion on $\Gamma$} 
	\label{app_Green}

	\begin{proposition}
		\label{prop_Green}
		Let $\mX $ be a general diffusion on the star-graph $\Gamma = (\{\bv\},E) $
		of edge-lengths $e\mapsto l_e $, defined by~\ref{item_analytical1}--\ref{item_analytical2} 
		on the probability space $\mc P_{\bx}\coloneqq (\Omega, \process{\bF_t},\Prob_{\bx}) $ such that $\mX_{0}=\bx $, $\Prob_{\bx} $-almost surely. 
		Let $U$ be the relatively compact star neighborhood defined by~\eqref{eq_form_neighborhood}. Assume the scale and speed regularity hypotheses of Proposition~\ref{prop_Dirichlet} and let $g\in C_b(\overline U)$. 
		Then,
		\begin{equation}
			\Esp_{\bx} \sqbraces{\int_{0}^{T_U} g(\mX_s) \vd s}
			= f(\bx), \quad \text{for all } \bx \in U,
		\end{equation}
		where $f$ is the function defined in~\eqref{eq_solution_Dirichlet}--\eqref{eq_constants_Dirichlet} with source term $-g$ and
		trivial boundary conditions ($a_e =0 $, for all $e\in E $). 
	\end{proposition}

\begin{proof}
	Let $f$ be the solution of
	Proposition~\ref{prop_Dirichlet}
	with source term $-g$ and zero boundary data.
	Thus $\Lop f=-g$ in $U$ and
	$f=0$ on $\partial U$.
	
	By the local generator characterization
	of the diffusion, the edgewise generalized
	differential equations and the vertex
	transmission condition imply that
	\[
	f(\mX_{t\wedge T_U})-f(\bx)
	+\int_0^{t\wedge T_U}g(\mX_s)\,\rd s
	\]
	is a martingale. Since $f$ and $g$ are
	bounded on $\overline U$, taking
	expectations gives
	\begin{equation}
		\label{eq_stopped_Dynkin_Green}
		\Esp_\bx[f(\mX_{t\wedge T_U})]
		=f(\bx)-
		\Esp_\bx\left[
		\int_0^{t\wedge T_U}g(\mX_s)\,\rd s
		\right].
	\end{equation}
	
	To justify passage to $T_U$, let $h$
	be the Dirichlet solution with source
	term $-1$ and zero boundary data.
	Its explicit Green representation
	\eqref{eq_solution_Dirichlet0}--%
	\eqref{eq_constants_Dirichlet0}
	shows that $h\ge0$ on $\overline U$.
	Applying the same stopped martingale
	identity to $h$ yields
	\[
	\Esp_\bx[t\wedge T_U]
	=h(\bx)
	-\Esp_\bx[h(\mX_{t\wedge T_U})]
	\le h(\bx)<\infty.
	\]
	By monotone convergence,
	$\Esp_\bx[T_U]\le h(\bx)<\infty$.
	
	Consequently, $T_U<\infty$ almost
	surely. Since $f$ and $g$ are bounded
	and $f=0$ on $\partial U$, dominated
	convergence in
	\eqref{eq_stopped_Dynkin_Green}
	gives
	\[
	f(\bx)=
	\Esp_\bx\left[
	\int_0^{T_U}g(\mX_s)\,\rd s
	\right].
	\]
\end{proof}

	\section{Conditioning on the embedded path} 
	\label{app_embedded}

\begin{lemma}
	\label{lem_conditioning}
	Let $Y_k=\mX_{\tau_k}$ be the vertex skeleton and
	$\mc B=\sigma(Y_k:k\ge0)$, as in
	Section~\ref{sec_embedding}.
	For $1\le k\le N$, let
	\[
	\mathscr E_k
	=(\mX_{(\tau_{k-1}+t)\wedge\tau_k})_{t\ge0}.
	\]
	For every $n\ge1$ and bounded measurable
	functions $F_1,\ldots,F_n$ of a stopped
	excursion path, we have
	\begin{align}
		&\Esp_\bx\left[
		\indic{N\ge n}
		\prod_{k=1}^n F_k(\mathscr E_k)
		\,\middle|\,\mc B\right]
		\notag\\
		&\quad=
		\indic{N\ge n}
		\prod_{k=1}^n
		\Esp_{Y_{k-1}}\left[
		F_k((\mX_{t\wedge\eta_{Y_{k-1}}})_{t\ge0})
		\,\middle|\,
		\mX_{\eta_{Y_{k-1}}}=Y_k
		\right].
	\end{align}
	The conditional expectations on the right
	are taken only for transitions of positive
	probability.
	
	The formula concerns transitions before
	or into $\mathcal A$; no conditional exit
	kernel is required after absorption.
	Moreover, conditional on $Y_0$, the initial
	segment before $\tau_0$ is independent of
	all subsequent embedded excursions.
\end{lemma}

\begin{proof}
	For bounded measurable functions $F_1,\ldots,F_n$
	and a sequence of vertices $y_0,\ldots,y_n$
	corresponding to allowed transitions, repeated
	application of the strong Markov property at
	$\tau_0,\tau_1,\ldots$ gives
	\begin{align}
		&\Esp_\bx\left[
		\prod_{k=1}^n F_k(\mathscr E_k)
		\indic{Y_0=y_0,\ldots,Y_n=y_n}
		\right]
		\notag\\
		&\quad=
		\mu_0(y_0)
		\prod_{k=1}^n
		\Esp_{y_{k-1}}\left[
		F_k((\mX_{t\wedge\eta_{y_{k-1}}})_{t\ge0})
		\indic{\mX_{\eta_{y_{k-1}}}=y_k}
		\right].
	\end{align}
	
	Taking $F_k\equiv1$ yields
	\[
	\Prob_\bx(Y_0=y_0,\ldots,Y_n=y_n)
	=\mu_0(y_0)
	\prod_{k=1}^n\fp_{y_{k-1},y_k}.
	\]
	Dividing the first identity by this
	probability gives the conditional product
	formula for cylinder events, since every
	realized transition has positive probability.
	
	A monotone-class argument extends the
	identity to the complete skeleton
	$\sigma$-algebra and all bounded measurable
	functions $F_1,\ldots,F_n$.
	Finally, conditional independence of the
	initial segment before $\tau_0$ follows
	from the strong Markov property at $\tau_0$.
\end{proof}

	\subsection*{Acknowledgments}
	This research was supported by the ANR project DREAMeS (ANR-21-CE46-0002-04) by the Centro de Modelamiento Matemático (CMM) BASAL fund FB210005 for center of excellence from ANID-Chile. 
	The work builds on techniques for 1D diffusions developed in~\cite{anagnostakis2023general}, a collaboration with the author’s PhD advisors, Antoine Lejay and Denis Villemonais.

	
\end{document}